\documentclass[11pt]{amsart}

\usepackage[english]{babel}
\usepackage{enumitem}
\usepackage[ansinew]{inputenc}
\usepackage{amsmath,amsthm}
\usepackage{hyperref}
\usepackage[T1]{fontenc}
\usepackage{lmodern}
\usepackage{cleveref}
\usepackage{amssymb}
\usepackage{mathtools}
\usepackage{amscd}
\usepackage{marginnote}
\usepackage[all]{xy} % XY-Matrix
\usepackage{dsfont}%einheitsmatrix
\usepackage{tikz-cd}
\usepackage{stmaryrd} % Eckige Klammern [[ \llbracket \rrbracket ]]

\setenumerate[1]{label=(\alph*)}
\setenumerate[2]{label=(\roman*)}

\theoremstyle{plain}
\newtheorem{thm}{Theorem}[section]
\newtheorem*{thm*}{Theorem}
\newtheorem{lem}[thm]{Lemma}
\newtheorem{cor}[thm]{Corollary}
\newtheorem{prop}[thm]{Proposition}

\theoremstyle{remark}
\newtheorem{rem}[thm]{Remark}

\theoremstyle{definition}

\newtheorem{defin}[thm]{Definition}

\DeclareFontFamily{OT1}{pzc}{}
\DeclareFontShape{OT1}{pzc}{m}{it}{<-> s * [1.10] pzcmi7t}{}
\DeclareMathAlphabet{\mathpzc}{OT1}{pzc}{m}{it}

\newcommand{\Dcal}{\mathcal{D}}
\newcommand{\Ecal}{\mathcal{E}}

\newcommand{\Fcal}{\mathcal{F}}
\newcommand{\Gcal}{\mathcal{G}}
\newcommand{\Hcal}{\mathcal{H}}
\newcommand{\Ical}{\mathcal{I}}
\newcommand{\Jcal}{\mathcal{J}}
\newcommand{\Lcal}{\mathcal{L}}
\newcommand{\Mcal}{\mathcal{M}}

\newcommand{\Ocal}{\mathcal{O}}
\newcommand{\Pcal}{\mathcal{P}}

\newcommand{\Ucal}{\mathcal{U}}

\newcommand{\Ycal}{\mathcal{Y}}

\newcommand{\QQ}{\mathbb{Q}}

\newcommand{\ZZ}{\mathbb{Z}}
\newcommand{\CC}{\mathbb{C}}
\newcommand{\HH}{\mathbb{H}}

\newcommand{\con}[1]{\nabla_{#1}}
\newcommand{\Po}{\mathcal{P}} %Poincare bundle
\newcommand{\Ed}{E^\vee} % dual elliptic curve
\newcommand{\id}{\mathrm{id}}
\newcommand{\righteq}{\stackrel{\sim}{\rightarrow}}
\newcommand{\pr}{\mathrm{pr}}

\newcommand{\Def}{\mathrm{Def.}}
\newcommand{\dd}{\mathrm{d}}
\newcommand{\om}{\underline{\omega}}
\newcommand{\HdR}[2]{\underline{H}^{#1}_{\mathrm{dR}}\left( #2 \right)}
\newcommand{\HdRabs}[2]{H^{#1}_{\mathrm{dR}}\left( #2 \right)}
\newcommand{\FKatoD}[2]{\prescript{}{D}{\mathrm{F}^{(#1)}_{#2}}}

\newcommand{\Ln}{\mathcal{L}_n}

\newcommand{\lnD}{l^D_n}

\newcommand{\triv}{\mathrm{triv}}
\newcommand{\HOM}{\underline{\mathrm{Hom}}} %sheaf hom
\newcommand{\Hom}{\mathrm{Hom}} %normal hom
\newcommand{\trans}{\mathrm{trans}} %sheaf hom
\newcommand{\inv}{\mathrm{inv}} %sheaf hom
\newcommand{\trv}{\tilde{\mathfrak{t}}}

\newcommand{\VIC}[1]{\mathrm{VIC}\left( #1 \right)}
\newcommand{\Und}[1]{\mathrm{U}^\dagger_n\left( #1 \right)}
\newcommand{\Uninf}{\mathrm{U}_n\left( \Ocal_{E^\dagger} \right)}
\newcommand{\gr}{\mathrm{gr}}
\newcommand{\Logn}{\mathrm{Log}^n_{\mathrm{dR}}}

\newcommand{\Log}[1]{\mathrm{Log}^{#1}_{\mathrm{dR}}}
\newcommand{\intHOM}{\underline{\mathrm{Hom}}} %sheaf hom
\newcommand{\idM}{\mathds{1}}
\newcommand{\res}{\mathrm{res}} 
\newcommand{\polD}{\mathrm{pol}_{D,\mathrm{dR}}}

\newcommand{\LnD}{L_n^D}
\newcommand{\KK}{K}
\newcommand{\scan}{s_{\mathrm{can}}}

\newcommand{\KS}{\mathrm{KS}} %Kodaira-Spencer
\newcommand{\thetaD}{\prescript{}{D}\theta}

\DeclareMathOperator{\Spec}{Spec}
\DeclareMathOperator{\Res}{Res}
\DeclareMathOperator{\Sym}{\underline{Sym}}
\DeclareMathOperator{\Inf}{Inf}
\DeclareMathOperator{\TSym}{\underline{TSym}}

\DeclareMathOperator{\Pic}{\underline{\mathrm{Pic}}^0_{E/S}}

\DeclareMathOperator{\Ext}{\mathrm{Ext}}

\title[The elliptic polylogarithm via the Poincar\'e bundle]{The algebraic de Rham realization of the elliptic polylogarithm via the Poincar\'e bundle}

\author{Johannes Sprang}
\email{johannes.sprang@mathematik.uni-regensburg.de }
\date{}

\hypersetup{
pdfauthor = {J. Sprang},
pdftitle = {The algebraic de Rham realization of the elliptic polylogarithm via the Poincar\'e bundle}
}

\begin{document}

\begin{abstract}
In this paper, we describe the algebraic de Rham realization of the elliptic polylogarithm for arbitrary families of elliptic curves in terms of the Poincar\'e bundle. Our work builds on previous work of Scheider and generalizes results of Bannai--Kobayashi--Tsuji and Scheider. As an application, we compute the de Rham Eisenstein classes explicitly in terms of certain algebraic Eisenstein series.
\end{abstract}

\maketitle

\section{Introduction}
In the groundbreaking paper \cite{beilinsonConj}, Beilinson has stated his important conjectures expressing special values of $L$-functions up to a rational factor in terms of motivic cohomology classes under the regulator map to Deligne cohomology. For finer integrality questions one has to consider additionally regulator maps to other cohomology theories. In order to study particular cases of these conjectures, one needs to construct such cohomology classes and understand their realizations. An important source of such cohomology classes are polylogarithmic cohomology classes and their associated Eisenstein classes.\par 
The elliptic polylogarithm has been defined by Beilinson and Levin in their seminal paper \cite{beilinson_levin}. It is a motivic cohomology class living on the complement of certain torsion sections of an elliptic curve $E\rightarrow S$. By specializing the elliptic polylogarithm along torsion sections one obtains the associated \emph{Eisenstein classes}. Eisenstein classes have been fruitfully applied for studying special values of $L$-functions of imaginary quadratic fields, e.g. in \cite{deninger} and \cite{kings_tamagawa}. They also proved to be an important tool for gaining a better understanding of $L$-functions of modular forms. They appear in the construction of Kato's celebrated Euler system and more recently in the important works of Bertolini--Darmon--Rotger and Kings--Loeffler--Zerbes.\par 
The aim of this work is to describe the algebraic de Rham realization of the elliptic polylogarithm for arbitrary families of elliptic curves. In the case of a single elliptic curve with complex multiplication, such a description has been given by Bannai--Kobayashi--Tsuji in \cite{BKT}. Building on this, Scheider has generalized this to arbitrary families of complex elliptic curves in his PhD thesis \cite{rene}. Even more importantly, he has given an explicit description of the de Rham logarithm sheaves in terms of the Poincar\'e bundle on the universal vectorial extension of the dual elliptic curve:
\begin{thm*}[Scheider, 2014]\footnote{For a more precise version of this theorem, we refer to \Cref{DR_thmRene} in the main body of the text.}
For an elliptic curve $E/S$ let $E^\dagger$ be the universal vectorial extension of the dual elliptic curve and $(\Po^\dagger,\con{\Po^\dagger})$ the Poincar\'e bundle with connection on $E\times_S E^\dagger$. Then
\[
	\Ln^\dagger:=(\pr_E)_*\left(\Po^\dagger|_{\Inf_e^n E^\dagger} \right)
\]
provides an explicit model for the (abstractly defined) $n$-th de Rham logarithm sheaf $\Logn$.
\end{thm*}
For a positive integer $D$, the de Rham polylogarithm is a pro-system of cohomology classes with values in the logarithm sheaves, more explicitly
\[
	\polD\in\varprojlim_n \HdRabs{1}{E\setminus E[D],\Logn}.
\]
Building on this, Scheider used certain theta functions of the Poincar\'e bundle to construct \emph{analytic} differential forms representing the de Rham polylogarithm class for families of complex elliptic curves. Here, he followed the approach of Bannai--Kobayashi--Tsuji where a similar construction has been given for elliptic curves with complex multiplication. For all arithmetic applications it is indispensable to have an explicit \emph{algebraic} representative of the polylogarithm class. In the CM case studied by Bannai--Kobayashi--Tsuji, it is possible to prove the algebraicity of the coefficient functions of the involved theta function using the algebraicity of the Hodge decomposition. This leads to a purely algebraic description of the de Rham polylogarithm for CM elliptic curves. Unfortunately, this approach fails for arbitrary families of elliptic curves and thus it does not apply to the situation studied by Scheider.\par 
In this work, we address this problem and construct \emph{algebraic} differential forms with values in the logarithm sheaves representing the de Rham polylogarithm for families of elliptic curves. For an elliptic curve $E/S$ and a positive integer $D$, we have defined a certain $1$-form with values in the Poincar\'e bundle $\scan^D\in\Po^\dagger\otimes \Omega^1_{E/S}$, c.f.~\cite{EisensteinPoincare}. This section $\scan^D$ is called \emph{the Kronecker section} and serves as a substitute for the analytic theta functions appearing in the work of Bannai--Kobayashi--Tsuji and Scheider. Scheider's theorem allows us to view 
\[
	\lnD:=(\pr_E)_*\left( \scan^D\Big|_{\Inf^n_e E^\dagger} \right)\in\Gamma(E\setminus E[D],\Ln^\dagger\otimes \Omega^1_{E/S}).
\]
as a $1$-form with values in the $n$-th de Rham logarithm sheaf. In a second step, we lift these relative $1$-forms to absolute $1$-forms
\[
	\LnD\in \Gamma(E\setminus E[D],\Ln^\dagger\otimes \Omega^1_{E}).
\]
 Now, our main result states that the pro-system $(\LnD)_n$ gives explicit algebraic representatives of the de Rham polylogarithm class:
 \begin{thm*}\footnote{For a more precise version of this theorem, we refer to \Cref{DR_Pol_thm} in the main body of the text.} Let $E/S$ be a family of elliptic curves over a smooth scheme $S$ over a field of characteristic zero. The $1$-forms $\LnD$ form explicit algebraic representatives of the de Rham polylogarithm class, i.e.:
\[
	([\LnD])_n=\polD \in \varprojlim_n \HdRabs{1}{E\setminus E[D],\Logn}.
\]
\end{thm*}
As a byproduct, we deduce explicit formulas for the de Rham Eisenstein classes in terms of certain holomorphic Eisenstein series.\par 
The results of this paper depend heavily on the results of the unpublished PhD thesis of Scheider. In particular, Scheider's explicit description of the de Rham logarithm sheaves in terms of the Poincar\'e bundle will play a fundamental role in this paper. Scheider's original proof of this result is long and involved. A substantial part of this paper is devoted to making the results of Scheider available to the mathematical community. At the same time, we will simplify the proof of Scheider's theorem considerably.\par 
Our main motivation for this work comes from the wish of gaining a better understanding of the syntomic realization of the elliptic polylogarithm. Syntomic cohomology can be seen as a $p$-adic analogue of Deligne cohomology and replaces Deligne cohomology in the formulation of the $p$-adic Beilinson conjectures.  Till now, we only understand the syntomic realization in the case of a single elliptic curve with complex multiplication \cite{BKT} as well as the specializations of the syntomic polylogarithm along torsion sections, i.e.~the syntomic Eisenstein classes \cite{bannai_kings}. Building on the results of this paper, we generalize the results of \cite{bannai_kings} and \cite{BKT} and provide an explicit description of the syntomic realization over the ordinary locus of the modular curve in \cite{Syntomic}. Here, it is indispensable to have explicit \emph{algebraic} representatives for the de Rham polylogarithm class.\par 
The polylogarithm can also be defined for higher dimensional Abelian schemes and is expected to have interesting arithmetic applications. While we have a good understanding of the elliptic polylogarithm, not much is known in the higher dimensional case. 
Combining the results of this paper with Scheider's results gives a conceptional understanding of the elliptic polylogarithm in terms of the Poincar\'e bundle. We expect that the general structure of the argument should allow the generalization to higher dimensional Abelian schemes. A good understanding of the de Rham realization is an essential step towards the realization in Deligne and syntomic cohomology.\par

\section*{Acknowledgement}
The results presented in this paper are part of my PhD thesis at the Universit\"at Regensburg \cite{PhD}. It is a pleasure to thank my advisor Guido Kings for his guidance during the last years. Further, I would like to thank Shinichi Kobayashi for all the valuable suggestions on my PhD thesis. I am also grateful for interesting discussions with Takeshi Tsuji in Lyon and Laurent Berger for his hospitality while visiting Lyon. The author would also like to thank the collaborative research centre SFB 1085 ``Higher Invariants'' by the Deutsche Forschungsgemeinschaft for its support. Last but not least, I would like to thank the referees for valuable comments and remarks.

\section{The de Rham logarithm sheaves}
The aim of this section is to define the pro-system of the de Rham logarithm sheaves. The de Rham logarithm sheaves satisfy a universal property among all unipotent vector bundles with integrable connections. In this section we will present the basic properties of the de Rham logarithm sheaves, these have been worked out by Scheider in his PhD thesis \cite{rene}.

\subsection{Vector bundles with integrable connections} 
The coefficients for algebraic de Rham cohomology are vector bundles with integrable connections. Let us start by with recalling some basic definitions on vector bundles with integrable connections. For details we refer to \cite[(1.0),(1.1)]{katz_monodromy}. For a smooth morphism $\pi:S\rightarrow T$ between smooth separated schemes of finite type over a field $\KK$  of characteristic $0$ let us denote by $\VIC{S/T}$ the category of vector bundles on $S$ with integrable $T$-connection and horizontal maps as morphisms. Since every coherent $\Ocal_S$-module with integrable $\KK$-connection is a vector bundle, the category $\VIC{S/\KK}$ is Abelian (cf.~\cite[\S2, Note 2.17]{berthelot_ogus}). The pullback along a smooth map $\pi:S\rightarrow T$ of smooth $\KK$-schemes induces an exact functor
\[
	\pi^* :\VIC{T/\KK}\rightarrow \VIC{S/\KK}.
\] 
By restricting the connection we get a forgetful map
\[
	\res_T\colon \VIC{S/\KK}\rightarrow \VIC{S/T}.
\]
To an object $\Fcal$ in $\VIC{S/T}$ we can associate a complex of $\pi^{-1}\Ocal_T$-modules
\[
	\Omega^\bullet_{S/T}(\Fcal)\colon \Fcal \rightarrow \Fcal\otimes_{\Ocal_S}\Omega^1_{S/T}\rightarrow ...
\]
called the \emph{algebraic de Rham complex}. The differentials in this complex are induced by the integrable connection. The relative algebraic de Rham cohomology for $\pi\colon S\rightarrow T$ is defined as
\[
	\HdR{i}{S/T,\Fcal}:=R^i\pi_*\left( \Omega^\bullet_{S/T}(\Fcal) \right).
\]
For $T=\Spec \KK$ we obtain the absolute algebraic de Rham cohomology $\HdRabs{i}{S,\Fcal}$. For $i=0,1$ the $i$-th de Rham cohomology can be seen as an extension group in the category $\VIC{S/\KK}$, i.e.
\[
	\HdRabs{i}{S,\Fcal}=\Ext^i_{\VIC{S/\KK}}(\Ocal_S,\Fcal),\quad \text{ for } i=0,1.
\]
For $\Fcal\in\VIC{X/T}$ and smooth morphisms $f: X\rightarrow S$ and $S\rightarrow T$ the relative de Rham cohomology
\[
	\HdR{i}{X/S,\Fcal}
\]
is canonically equipped with an integrable $T$-connection called \emph{Gauss--Manin} connection.

\subsection{Definition of the de Rham logarithm sheaves} Let $\pi:E\rightarrow S$ be an elliptic curve over a smooth separated $K$-scheme of finite type. Let us write $\Hcal:=\HdR{1}{E/S}^\vee$ for the dual of the relative de Rham cohomology and $\Hcal_E:=\pi^*\Hcal$ for its pullback to the elliptic curve. The Gauss--Manin connection equips $\Hcal_E$ with an integrable $\KK$-connection. The group $\Ext^1_{\VIC{E/\KK}}(\Ocal_E,\Hcal_E)$ classifies isomorphism classes $[\Fcal]$ of extensions
\begin{equation}\label{eq_ext1}
\begin{tikzcd}
	0\ar[r] & \Hcal_E \ar[r] & \Fcal \ar[r] & \Ocal_E \ar[r] & 0
\end{tikzcd}
\end{equation}
in the category $\VIC{E/\KK}$. In general such an extension will have non-trivial automorphisms. In the case where the extension \eqref{eq_ext1} splits horizontally after pullback along the unit section $e\colon S\rightarrow E$ we can rigidify the situation by fixing a splitting, i.e. an isomorphism
\[
\begin{tikzcd}
	0\ar[r] & \Hcal \ar[r] & \Hcal\oplus \Ocal_S\ar[d,"\cong"] \ar[r] & \Ocal_S \ar[r] & 0\\
	0\ar[r] & e^*\Hcal_E \ar[r]\ar[u,equals] & e^*\Fcal \ar[r] & \Ocal_S \ar[r]\ar[u,equals] & 0	
\end{tikzcd}
\]
in the category $\VIC{S/\KK}$. Such a horizontal splitting is uniquely determined by the image of $1\in\Gamma(S,\Ocal_S)$ under the splitting. In other words, the group
\[
	\ker\left( e^*\colon\Ext^1_{\VIC{E/\KK}}(\Ocal_E,\Hcal_E)\rightarrow \Ext^1_{\VIC{S/\KK}}(\Ocal_S,\Hcal)  \right)
\]
classifies isomorphism classes of pairs $[\Fcal,s]$ consisting of an extension of $\Ocal_E$ by $\Hcal_E$ in the category $\VIC{E/\KK}$ together with a horizontal section $s\in\Gamma(S,e^*\Fcal)$ mapping to $1$ under $e^*\Fcal\rightarrow\Ocal_S$. A pair $[F,s]$ is uniquely determined by its extension class up to unique isomorphism.\par 

 The Leray spectral sequence in de Rham cohomology 
\[
	E_2^{p,q}=\HdRabs{p}{S/\KK,\HdR{q}{E/S,\Hcal_E}}\implies E^{p+q}=\HdRabs{p+q}{E/\KK,\Hcal_E}
\]
gives a split short exact sequence
\[
\begin{tikzcd}[column sep=small]
	0\ar[r] & \Ext^1_{\VIC{S/\KK}}(\Ocal_S,\Hcal)\ar[r,shift left,"\pi^*"] & \ar[l,shift left,"e^*"]\Ext^1_{\VIC{E/\KK}}(\Ocal_E,\Hcal_E) \ar[r] & \Hom_{\VIC{S/\KK}}(\Ocal_S,\Hcal\otimes_{\Ocal_S}\Hcal^\vee)\ar[r] & 0.
\end{tikzcd}
\]
\begin{defin}\label{def_log}
	Let $[\mathrm{Log}^1_{\mathrm{dR}},\idM^{(1)}]$ be the unique extension class corresponding to $\id_{\Hcal}$ under the isomorphism
	\[
		\ker\left( e^*\colon\Ext^1_{\VIC{E/\KK}}(\Ocal_E,\Hcal_E)\rightarrow \Ext^1_{\VIC{S/\KK}}(\Ocal_S,\Hcal)  \right)\cong\Hom_{\VIC{S/\KK}}(\Ocal_S,\Hcal\otimes_{\Ocal_S}\Hcal^\vee).
	\]
	This pair $(\mathrm{Log}^1_{\mathrm{dR}},\idM^{(1)})$ is uniquely determined up to unique isomorphism. This pair is called the \emph{first de Rham logarithm sheaf}.
\end{defin}
The tensor product in the category $\VIC{E/\KK}$ allows us to define for $n\geq 0$ an integrable connection on the $n$-th tensor symmetric powers
\[
	\Logn:=\TSym^n \mathrm{Log}^1_{\mathrm{dR}}.
\] 
For details on symmetric tensors, let us for example refer to \cite[Ch. IV,\S 5]{bourbaki_algebra2}. Recall that the $n$-th tensor symmetric power is the sheaf of invariants of the $n$-th tensor power under the action of the symmetric groups $S_n$. The shuffle product defines a ring structure on the tensor symmetric powers and we obtain a graded ring 
\[
	\bigoplus_{k\geq 0} \TSym^k \mathrm{Log}^1_{\mathrm{dR}}.
\]
The map
\[
	\mathrm{Log}^1_{\mathrm{dR}}\rightarrow \TSym^k \mathrm{Log}^1_{\mathrm{dR}},\quad x\mapsto x^{[k]}:=\underbrace{x\otimes...\otimes x}_{k}
\]
defines divided powers on the graded algebra of symmetric tensors. In particular, the horizontal section $\idM^{(1)}$ induces a horizontal section
\[
	\idM^{(n)}:=(\idM^{(1)})^{[n]} \in \Gamma(S,e^*\Logn).
\]
This allows us to define the \emph{$n$-th de Rham logarithm sheaves} as the pair $(\Logn,\idM^{(n)})$. Later we will see that the \emph{$n$-th de Rham logarithm sheaf} is uniquely determined by a universal property, c.f.~\Cref{log_univ}. The horizontal epimorphism
$\mathrm{Log}^1_{\mathrm{dR}}\twoheadrightarrow \Ocal_E$ induces horizontal transition maps
\[
	\Logn \twoheadrightarrow \mathrm{Log}^{n-1}_{\mathrm{dR}}.
\]
The transition maps allow us to define a descending filtration
\[
	A^0\Logn=\Logn\supseteq A^1\Logn \supseteq ... \supseteq A^{n+1}\Logn=0
\]
by sub-objects
\[
	A^i\Logn:= \ker\left( \Logn\twoheadrightarrow \mathrm{Log}^{n-i}_{\mathrm{dR}} \right),\quad i=1,...,n. 
\]
in the category $\VIC{E/\KK}$. The graded pieces of this filtration are given by
\[
	\gr^i_A \Logn=A^{i}\Logn/A^{i+1}\Logn=\TSym^i \Hcal_E.
\]

\subsection{The cohomology of the de Rham logarithm sheaves}
For the definition of the de Rham polylogarithm we will need a good understanding of the cohomology of the logarithm sheaves. The necessary computations are the same as for other realizations. In the de Rham realization, they can also be found in Scheider's PhD thesis \cite[\S 1.2]{rene}. Let us briefly recall the arguments for the convenience of the reader.
\begin{prop}\label{prop_cohomLog}
	For $i=0,1$ the transition map $\Logn\twoheadrightarrow \mathrm{Log}^{n-1}_{\mathrm{dR}}$ induces the zero morphism
	\[
		\HdR{i}{E/S,\Log{n}}\rightarrow \HdR{i}{E/S,\Log{n-1}}.
	\]
	For $i=2$ we have isomorphisms
	\[
		\HdR{2}{E/S,\Log{n}}\righteq \HdR{2}{E/S,\Log{n-1}}\righteq ... \righteq \Ocal_S.
	\]
\end{prop}
\begin{proof} This is a classical result due to Beilinson and Levin \cite{beilinson_levin}. For the de Rham realization see also \cite[Thm 1.2.1.]{rene}\label{prop_transition}. For the convenience of the reader let us include a proof. For $n\geq 1$ consider the long exact sequence\\
\begin{tikzpicture}[descr/.style={fill=white,inner sep=1.5pt}]
        \matrix (m) [
            matrix of math nodes,
            row sep=2em,
            column sep=2.5em,
            text height=1.5ex, text depth=0.25ex
        ]
        { 0 & \HdR{0}{E/S,\Hcal_E\otimes\Log{n-1}} & \HdR{0}{E/S,\Log{n}} & \HdR{0}{E/S} & \\
            & \HdR{1}{E/S,\Hcal_E\otimes\Log{n-1}} & \HdR{1}{E/S,\Log{n}} & \HdR{1}{E/S} & \\
            & \HdR{2}{E/S,\Hcal_E\otimes\Log{n-1}} & \HdR{2}{E/S,\Log{n}} & \HdR{2}{E/S} & 0 \\
        };

        \path[overlay,->, font=\scriptsize,>=latex]
        (m-1-1) edge (m-1-2)
        (m-1-2) edge (m-1-3)
        (m-1-3) edge (m-1-4)
        (m-1-4) edge[out=355,in=175,black] node[descr,yshift=0.3ex] {$\delta^0$} (m-2-2)
        (m-2-2) edge (m-2-3)
        (m-2-3) edge (m-2-4)
        (m-2-4) edge[out=355,in=175,black] node[descr,yshift=0.3ex] {$\delta^1$} (m-3-2)
        (m-3-2) edge (m-3-3)
        (m-3-3) edge (m-3-4)
        (m-3-4) edge (m-3-5);
\end{tikzpicture}\\
associated to
\[
	0\rightarrow \Hcal_E\otimes\Log{n-1} \rightarrow \Log{n} \rightarrow \Ocal_E\rightarrow 0.
\]
Let us first consider the case $n=1$: It follows from the defining property of the extension class of $\Log{1}$ that the map $\delta^0\colon \Ocal_S=\HdR{0}{E/S}\rightarrow \HdR{1}{E/S,\Hcal_E}=\Hcal\otimes\Hcal^\vee$ maps $1\in\Gamma(S,\Ocal_S)$ to $\id_\Hcal\in \Gamma(S,\Hcal\otimes\Hcal^\vee)$, i.e.
\begin{equation}\label{eq_delta0}
	\delta^0(1)=\id_\Hcal
\end{equation}
In particular, we deduce that $\delta^0$ is injective. Again by \eqref{eq_delta0}, we obtain that
\begin{equation}\label{eq_cup1}
\HdR{1}{E/S}\cong \HdR{0}{E/S}\otimes\HdR{1}{E/S}\xrightarrow{\delta^0\otimes\id} \HdR{1}{\Hcal_E}\otimes \HdR{1}{E/S}\xrightarrow{\cup} \HdR{2}{\Hcal_E}
\end{equation}
coincides with the map 
\[
	\HdR{1}{E/S}\rightarrow \HOM_{\Ocal_S}(\HdR{1}{E/S},\Ocal_S)\cong\HdR{2}{\Hcal_E} , \quad x\mapsto (y\mapsto x\cup y).
\]
Since the cup product
\[
	\cup\colon \HdR{1}{E/S}\otimes \HdR{1}{E/S}\rightarrow \HdR{2}{E/S}
\]
defines a perfect pairing, we deduce that \eqref{eq_cup1} is an isomorphism. By the compatibility of the cup product with the connecting homomorphism
\[
	\begin{tikzcd}
		\HdR{0}{E/S}\otimes\HdR{1}{E/S}\ar[r,"\delta^0\otimes \id"]\ar[d,"\cup"] & \HdR{1}{\Hcal_E}\otimes \HdR{1}{E/S}\ar[d,"\cup"] \\
		\HdR{1}{E/S}\ar[r,"\delta^1"]  &  \HdR{2}{\Hcal_E}
	\end{tikzcd}
\]
we deduce that $\delta^1$ is an isomorphism. The fact that $\delta^1$ is an isomorphism implies that 
\[
	\HdR{2}{E/S,\Log{1}}\rightarrow \HdR{2}{E/S}
\]
is an isomorphism and that
\[
	\HdR{1}{E/S,\Log{1}}\rightarrow \HdR{1}{E/S}
\]
is zero. The injectivity of $\delta^0$ implies that
\[
	\HdR{0}{E/S,\Log{1}}\rightarrow \HdR{0}{E/S}
\]
is zero. This settles the case $n=1$. Let us now proceed by induction. For $n\geq 2$ let us assume that the claim has been proven for the transition map $\Log{n-1}\rightarrow\Log{n-2}$. The morphism of exact complexes associated to the map of short exact sequences
\[
	\begin{tikzcd}
	0\ar[r] &  \Hcal_E\otimes\Log{n-1}\ar[r]\ar[d] & \Log{n} \ar[r]\ar[d] & \Ocal_E  \ar[r]\ar[d] & 0 \\
	0\ar[r] &  \Hcal_E\otimes\Log{n-2}\ar[r] & \Log{n-1} \ar[r] & \Ocal_E  \ar[r] & 0 
	\end{tikzcd}
\]
splits up into pieces by the induction hypothesis. Indeed, note that $\Log{n}\rightarrow\Ocal_E$ factors through $\Log{n}\rightarrow\Log{n-1}\rightarrow... \rightarrow \Ocal_E$ and thus induces the zero map in cohomological degree $0$ and $1$ and an isomorphism in degree $2$. Using this, the above diagram of short exact sequences induces the following diagrams with exact rows:
\[
	\begin{tikzcd}
		0\ar[r] &  \HdR{0}{E/S,\Hcal_E\otimes\Log{n-1}}\ar[r,"\cong"]\ar[d,"0\text{ (by IH)}"] & \HdR{0}{E/S,\Log{n}} \ar[d,"(1)"]\\
		0\ar[r] &  \HdR{0}{E/S,\Hcal_E\otimes\Log{n-2}}\ar[r,"\cong"] & \HdR{0}{E/S,\Log{n-1}}
	\end{tikzcd}
\]
\[
	\begin{tikzcd}
		 \HdR{0}{E/S}\ar[r,hook]\ar[d,equals] & \HdR{1}{E/S,\Hcal_E\otimes\Log{n-1}} \ar[r,two heads]\ar[d,"0\text{ (by IH)}"] & \HdR{1}{E/S,\Log{n}}  \ar[d,"(2)"] \\
		 \HdR{0}{E/S}\ar[r,hook] & \HdR{1}{E/S,\Hcal_E\otimes\Log{n-2}} \ar[r,two heads] & \HdR{1}{E/S,\Log{n-1}} ,
	\end{tikzcd}
\]
and
\[
	\begin{tikzcd}
		 \HdR{1}{E/S} \ar[r,hook]\ar[d,equals] & \HdR{2}{E/S,\Hcal_E\otimes\Log{n-1}} \ar[r]\ar[d,"\cong\text{ (by IH)}"] & \HdR{2}{E/S,\Log{n}}  \ar[r,two heads]\ar[d,"(3)"] & \HdR{2}{E/S} \ar[d,equals]\\
		 \HdR{1}{E/S} \ar[r,hook] & \HdR{2}{E/S,\Hcal_E\otimes\Log{n-2}} \ar[r] & \HdR{2}{E/S,\Log{n-1}} \ar[r,two heads] & \HdR{2}{E/S}.
	\end{tikzcd}
\]
The maps denoted by (IH) are zero respectively isomorphisms by the induction hypothesis. From the  commutativity of the diagrams we deduce that the transition maps in $(1)$ and $(2)$ are zero maps while $(3)$ is an isomorphism, as desired.

\end{proof}

\subsection{The universal property of the de Rham logarithm sheaves}
In this subsection we will prove the universal property of the de Rham logarithm sheaves among all unipotent vector bundles with integrable connection. Recall that we have a descending filtration $A^\bullet\Logn$ on the $n$-th logarithm sheaves satisfying
\[
	\gr^i_A\Logn=\TSym^i\Hcal_E=\pi^* \TSym^i \Hcal.
\]
In particular, all sub-quotients are given by pullback. Motivated by this property we state the following definition:
\begin{defin}\label{def_unipotentsheaves}
	Let $S\rightarrow T$ be a smooth morphism of smooth separated $\KK$-schemes and $E/S$ an elliptic curve.
	\begin{enumerate}
		\item An object $\Ucal$ in  $\VIC{E/T}$ is called unipotent of length $n$ for $E/S/T$ if there exists a descending filtration in the category $\VIC{E/T}$
		\[
			\Ucal=A^0\Ucal \supseteq A^1\Ucal \supseteq ... \supseteq A^{n+1}\Ucal=0 
		\]
		such that for all $0\leq i\leq n$, $\gr^i_A  \Ucal=A^{i}\Ucal/A^{i+1}\Ucal=\pi^* Y_i$ for some $Y_i\in \VIC{S/T}$.
		\item Let $U^\dagger_n(E/S/T)$ be the full subcategory of objects of $\VIC{E/T}$ which are unipotent of length $n$ for $E/S/T$. For the case $S=T$ let us write $U^\dagger(E/S):=U^\dagger(E/S)$ for simplicity.
	\end{enumerate}
\end{defin}% TODO start here
For the moment we will consider the absolute case, i.e. $T=\Spec \KK$. Later, we will naturally by concerned with the relative case, i.e. the case $T=S$.
For $F,G\in\VIC{E/K}$ there is a natural connection on the sheaf of homomorphisms of the underlying modules. Let us write $\intHOM(F,G)$ for the internal-Hom in the category $\VIC{S/K}$. Let us also introduce the notation  $\HOM^{\mathrm{hor}}_{E/S}(F,G)$ for the sheaf of horizontal morphisms relative $S$. By the definition of the connection on the internal-Hom sheaves, we see that $\HOM^{\mathrm{hor}}_{E/S}(F,G)$ is the subsheaf of $S$-horizontal sections of $\intHOM(F,G)$. We can rephrase this as follows:
\[
	\pi_*\HOM^{\mathrm{hor}}_{E/S}(F,G)=\HdR{0}{E/S,\intHOM(F,G)}.
\]
In particular, the Gauss--Manin connection gives a $\KK$-connection on $\pi_*\HOM^{\mathrm{hor}}_{E/S}(F,G)$. With this observation we can finally characterize the logarithm sheaves through a universal property:
\begin{prop}\cite[Thm~1.3.6]{rene} \label{log_univ}
The pair $(\Logn, \idM^{(n)})$ is the unique pair, consisting of a unipotent object of $\Und{E/S/\KK}$ together with a horizontal section along $e$, such that for all $\Ucal\in \Und{E/S/\KK}$ the map
\[
	\pi_* \HOM^{\mathrm{hor}}_{E/S}(\Logn,\Ucal)\rightarrow e^*\Ucal,\quad f\mapsto (e^*f)(\idM^{(n)})
\]
is an isomorphism in $\VIC{S/\KK}$.
\end{prop}
\begin{proof}
For the convenience of the reader let us sketch the proof. We have for $0\leq i\leq 2$ a canonical horizontal isomorphism
\begin{equation}\label{eq_duality}
	\HdR{2-i}{E/S,\left(\Logn\right)^\vee}\righteq \HdR{i}{E/S,\Logn}^\vee
\end{equation}
induced by the perfect cup product pairing
\[
	\HdR{2-i}{E/S,\left(\Logn\right)^\vee}\otimes_{\Ocal_S}\HdR{i}{E/S,\Logn}\rightarrow \HdR{2}{E/S}\righteq  \Ocal_S.
\]
We prove the result by a double induction over $0\leq k\leq n$, where $k$ is the length of the shortest unipotent filtration of the object $\Ucal$. For $k=n=0$, we have $\Ucal=\pi^*Z$ for some object $Z$ of $\VIC{S/\KK}$. Indeed, we have
\[
	\pi_* \HOM^{\mathrm{hor}}_{E/S}(\Log{0},\pi^*Z)\cong \pi_* \HOM^{\mathrm{hor}}_{E/S}(\Ocal_E,\pi^*Z)\cong \HOM(\Ocal_S,Z)\cong Z
\]
and the map is given by $f\mapsto e^*f(1)=e^*f(\idM^{(0)})$. For the case $0=k<n$ observe, that the transition maps
\[
	\HdR{0}{E/S,\left(\Logn\right)^\vee}\righteq \HdR{0}{E/S,\left(\Log{n-1}\right)^\vee}
\]
are isomorphisms by the above perfect pairing and \Cref{prop_cohomLog}. We deduce the claim for $0=k<n$ by the commutative diagram
\[
\begin{tikzcd}
	\HdR{0}{E/S,\left(\Logn\right)^\vee} \ar[r,"\simeq"]\ar[d,"\simeq"] & \pi_* \HOM^{\mathrm{hor}}_{E/S}(\Log{n},\pi^*Z)\ar[r]\ar[d] & Z\ar[d,equal]\\
	\HdR{0}{E/S,\left(\Log{0}\right)^\vee} \ar[r,"\simeq"] & \pi_* \HOM^{\mathrm{hor}}_{E/S}(\Log{0},\pi^*Z)\ar[r,"\simeq"] & Z.
\end{tikzcd}
\]
Here, let us observe, that $\idM^{(n)}$ maps to $\idM^{(0)}$ under the transition map. Let us now consider the case $0<k\leq n$ and assume that the case $k-1\leq n$ has already been settled. For an unipotent object $\Ucal$ of length $k$ we have an horizontal exact sequence
\[
	\begin{tikzcd}
		0\ar[r] & \pi^* Z \ar[r] & \Ucal \ar[r] & \Ucal/\pi^*Z \ar[r]  & 0
	\end{tikzcd}
\]
with $\Ucal/\pi^*Z$ in $U_{k-1}^\dagger(E/S/\KK)$. This sequence induces a long exact sequences in $\VIC{S/\KK}$. Let us first show that the connecting homomorphism
\[
  \delta^{0}_{(n)}\colon \HdR{0}{E/S,\intHOM(\Logn,\Ucal/\pi^*Z)}\rightarrow \HdR{0}{E/S,\intHOM(\Logn,\pi^*Z)}
\]
is trivial for all $n\geq k$. Let us consider the following commutative diagram, where the right vertical map is induced by the transition maps of the logarithm sheaves:
\[
	\begin{tikzcd}
		\HdR{0}{E/S,\intHOM(\mathrm{Log}_{dR}^{n-1},\Ucal/\pi^*Z)}\ar[d]\ar[r] &e^*(\Ucal/\pi^*Z)\ar[d,equals] \\
		\HdR{0}{E/S,\intHOM(\Logn,\Ucal/\pi^*Z)}\ar[r] & e^*(\Ucal/\pi^*Z),
	\end{tikzcd}
\]
By the induction hypothesis, the horizontal maps in this diagram are isomorphisms. We deduce that the map
\[
	\HdR{0}{E/S,\intHOM(\Log{n-1},\Ucal/\pi^*Z)}\rightarrow \HdR{0}{E/S,\intHOM(\Log{n},\Ucal/\pi^*Z)}
\]
is an isomorphism, too. On the other hand, the transition maps
\begin{equation}\label{eq_transition_zero}
	\HdR{1}{E/S,\intHOM(\Log{n-1},\pi^*Z)}\rightarrow \HdR{1}{E/S,\intHOM(\Log{n},\pi^*Z)}
\end{equation}
are identified with the dual of the transition maps
\[
	\HdR{1}{E/S,\Log{n-1}}^\vee \otimes Z \rightarrow \HdR{1}{E/S,\Log{n}}^\vee \otimes Z.
\]
under the perfect cup product pairing. Thus, \cref{prop_transition} implies that \eqref{eq_transition_zero} is the zero morphism. The connecting homomorphisms $\delta_{(n)}^0$ and $\delta_{(n-1)}^0$ fit in the following commutative diagram, where the vertical maps are induced by the transition maps of the logarithm sheaves:
\[
	\begin{tikzcd}
		\HdR{0}{E/S,\intHOM(\Log{n-1},\Ucal/\pi^*Z)} \ar[r,"\delta_{(n-1)}^0"]\ar[d,"\cong"] &  \HdR{1}{E/S,\intHOM(\Log{n-1},\pi^*Z)}\ar[d,"0"]\\
		\HdR{0}{E/S,\intHOM(\Logn,\Ucal/\pi^*Z)} \ar[r,"\delta_{(n)}^0"] &  \HdR{1}{E/S,\intHOM(\Logn,\pi^*Z)}.
	\end{tikzcd}
\]
Above, we have already shown that the left vertical map is an isomorphism, while the right vertical map is zero. This shows the vanishing of the connecting homomorphism $\delta_{(n)}^0$. Now, the claim for $0<k\leq n$ is easily deduced from the induction hypothesis. Indeed, we get the following commutative diagram with vertical exact sequences:
\[
\begin{tikzcd}
	0 \ar[d] & 0\ar[d]  \\
	\HdR{0}{E/S,\intHOM(\Logn,\pi^*Z)} \ar[d] \ar[r,"\simeq"] & Z \ar[d] \\
	 \HdR{0}{E/S,\intHOM(\Logn,\Ucal)} \ar[d] \ar[r] & e^*\Ucal\ar[d]  \\
	  \HdR{0}{E/S,\intHOM(\Logn,\Ucal/\pi^*Z)} \ar[d] \ar[r,,"\simeq"] & e^*(\Ucal/\pi^*Z)\ar[d]  \\
	  0 & 0
\end{tikzcd}
\]
Here, the first and the last horizontal maps are isomorphisms by induction. The left-exactness of the first column follows from the vanishing of the connecting homomorphism $\delta_{(n)}^0$. We deduce the desired isomorphism in the middle.
\end{proof}

\begin{rem}
	Another way to formulate the universal property is as follows. Consider the category consisting of pairs $(\Ucal,s)$ with $\Ucal\in \Und{E/S/\KK}$ and a fixed horizontal section $s\in\Gamma(S,e^*\Ucal)$. Morphisms are supposed to be horizontal and respect the fixed section after pullback along $e$. Then, the universal property reformulates as the fact that this category has an initial object. This initial object is $(\Logn, \idM^{(n)})$.
\end{rem}

\section{The Definition of the polylogarithm in de Rham cohomology}\label{ch:DR_secdefpol}
Let us briefly recall the definition of the de Rham cohomology class of the polylogarithm following \cite[Chapter 1.5]{rene}. Let us fix a positive integer $D$. Let us define the sections $1_{e},1_{E[D]}\in\Gamma(E[D],\Ocal_{E[D]})$ as follows: Let $1_{E[D]}$ correspond to $1\in\Ocal_{E[D]}$ and $1_{e}$ correspond to the section which is zero on $E[D]\setminus\{e\}$ and $1$ on $\{e\}$. The localization sequence in de Rham cohomology for the situation
\begin{equation}\label{DR_Pol_eq1}
	\begin{tikzcd}
		U_D:=E\setminus E[D] \ar[r,hook,"j_D"]\ar[rd,swap,"\pi_{U_D}"] & E\ar[d,"\pi"] & E[D]\ar[ld,"\pi_{E[D]}"] \ar[l,swap,"i_D"]\\
		& S &
	\end{tikzcd}
\end{equation}
combined with the vanishing results (c.f.~\Cref{prop_transition}) gives the following.
\begin{lem}[{\cite[\S 1.5.2, Lemma 1.5.4]{rene}}]
	Let us write $\Hcal_{E[D]}:=\pi_{E[D]}^*\Hcal$ and $\Hcal_{U_D}:=\pi_{U_D}^*\Hcal$. The localization sequence in de Rham cohomology for \eqref{DR_Pol_eq1} induces an exact sequence:
	\[
		\begin{tikzcd}[column sep=small]
			0\ar[r] & \varprojlim_n \HdRabs{1}{U_D/\KK, \Logn}\ar[r,hook,"\Res"] & \prod_{k=0}^\infty \HdRabs{0}{E[D]/\KK,\Sym^k\Hcal_{E[D]}}\ar[r,"\sigma"] & K\ar[r] & 0
		\end{tikzcd}
	\]
	If we view the horizontal section $D^2 \cdot 1_{e}-1_{E[D]}\in \Gamma(S,\Ocal_{E[D]})$ as sitting in degree zero of
	\[
	\prod_{k=0}^{\infty} \HdR{0}{E[D]/\KK,\Sym^k \Hcal_{E[D]}},
	\]
	it is contained in the kernel of the augmentation map $\sigma$.
\end{lem}
\begin{proof}
	For the convenience of the reader let us recall the construction of the short exact sequence. The localization sequence and the vanishing of $\varprojlim_n \HdRabs{1}{E,\Logn}=0$ gives
	\[
	\begin{tikzcd}[column sep=small]
		0\ar[r] & \varprojlim\limits_n \HdRabs{1}{U_D,\Logn} \ar[r,hook,"\Res"] & \varprojlim\limits_n \HdRabs{0}{E[D],i_D^*\Logn} \ar[r] & \varprojlim\limits_n \HdRabs{2}{E,\Logn} \ar[r] & 0.
	\end{tikzcd}
	\]
	Now, the exact sequence in the claim follows by \Cref{prop_transition} and the isomorphism
	\[
		i_D^*\Logn\righteq i_D^*[D]^*\Logn=\pi_{E[D]}^*e^*\Logn\righteq \bigoplus_{k=0}^n \Sym^k \Hcal_{E[D]}.
	\]
\end{proof}
\begin{defin}
	Let $\polD=(\polD^n)_{n\geq0}\in\varprojlim_n \HdRabs{1}{U_D/\KK,\Logn}$ be the unique pro-system mapping to $D^21_{e}-1_{E[D]}$ under the residue map. We call $\polD$ \emph{the ($D$-variant) of the elliptic polylogarithm}.
\end{defin}

\begin{rem}
	Let us write $U:=E\setminus\{e\}$. The \emph{classical polylogarithm}  in de Rham cohomology 
	\[
		(\mathrm{pol}_{\mathrm{dR}}^n)_{n\geq 0}\in \varprojlim_n \HdRabs{1}{U/K, \Hcal_E^\vee\otimes_{\Ocal_E}\Logn }
	\]	
	 is defined as the unique element mapping to $\id_{\Hcal}$  under the isomorphism
	\[
		\varprojlim_n \HdRabs{1}{U/K, \Hcal_E^\vee\otimes_{\Ocal_E}\Logn } \righteq \prod_{k=1}^\infty \HdRabs{0}{S/\QQ, \Hcal^\vee\otimes_{\Ocal_S}\Sym^k \Hcal}.
	\]
	This isomorphism comes from the localization sequence for $U:=E\setminus\{e\}\hookrightarrow E$. For details we refer to \cite[\S 1.5.1]{rene}. Indeed, there is not much difference between the classical polylogarithm and its $D$-variant. For a comparison of both we refer to \cite[\S 1.5.3]{rene}.
\end{rem}

\section{The de Rham logarithm sheaves via the Poincar\'e bundle}
	In his PhD thesis Scheider gave an explicit model for the de Rham logarithm sheaves constructed out of the Poincar\'e bundle \cite[Theorem 2.3.1]{rene}. Since the material has never been published, we will recall his approach to the de Rham logarithm sheaves via the Poincar\'e bundle. The main result of this section is due to Scheider, but we provide a considerably shorter proof of this theorem.
	\subsection{The geometric logarithm sheaves}	
	Let $\pi:E\rightarrow S$ be an elliptic curve over a separated locally Noetherian base scheme $S$. Let us recall the definition of the Poincar\'e bundle and thereby fix some notation. A \emph{rigidification} of a line bundle $\Lcal$ on $E$ over $S$ is an isomorphism
\[
	r:e^*\Lcal \righteq \Ocal_S.
\]
A morphism of rigidified line bundles is a morphism of line bundles respecting the rigidification. The dual elliptic curve $E^\vee$ represents the functor
\[
	T\mapsto \mathrm{Pic}^0(E_T/T):=\{ \text{iso. classes of rigidified line bundles } (\Lcal,r) \text{ of degree $0$ on } E_T/T\}
\]
on the category of $S$-schemes. Because a rigidified line bundle does not have any non-trivial automorphisms, there is a universal rigidified line bundle $(\Po,r_0)$ called the \emph{Poincar\'e bundle} over $E\times_S \Ed$. By interchanging the roles of $E$ and $\Ed$, we get a unique trivialization $s_0\colon (e\times\id)^*\Po\righteq \Ocal_{\Ed}$ and we call $(\Po,r_0,s_0)$ the \emph{bi-rigidified Poincar\'e bundle}. Similarly, let us consider the group valued functor
\[
	T\mapsto \mathrm{Pic}^\dagger(E_T/T):=
	\left.\begin{cases}
	\text{iso. classes of rigidified line bundles } (\Lcal,r,\nabla) \text{ of degree $0$ on } E_T/T \\
	\text{with an integrable $T$-connection $\nabla\colon \Lcal\rightarrow \Lcal\otimes_{\Ocal_{E_T}}\Omega^1_{E_T/T}$ on $\Lcal$}
	\end{cases}	\right\}.
\]
This functor is representable by an $S$-group scheme $E^\dagger$. By forgetting the connection, we obtain an epimorphism $q^\dagger\colon E^\dagger\rightarrow E^\vee$ of group schemes over $S$. The pullback $\Po^\dagger:=(q^\dagger)^*\Po$ is equipped with a unique integrable $E^\dagger$-connection
	\[
		\con{\Po^\dagger}: \Po^\dagger \rightarrow \Po^\dagger\otimes\Omega^1_{E\times_S E^\dagger/E^\dagger}
	\]
	making $(\Po^\dagger,\con{\Po^\dagger},r_0)$ universal among all rigidified line bundles with integrable connection. Let us mention that the group scheme $E^\dagger$ satisfies another universal property: It sits in a short exact sequence
	\[
		0\rightarrow V(\omega_{E/S})\rightarrow E^\dagger \rightarrow E^\vee \rightarrow 0
	\]
	with $V(\omega_{E/S})$ the vector group over $S$ associated with $\omega_{E/S}$ and every other such vectorial extension of $E^\vee$ is a pushout of this extension. This explains why $E^\dagger$ is called \emph{universal vectorial extension} of $E^\vee$. For a more detailed discussion on the universal vectorial extension and its properties, let us refer to the first chapter of the book of Mazur--Messing \cite{mazur_messing}.\par 
	 Let us denote the inclusions of the infinitesimal thickenings of $e$ in $E^\dagger$ resp. $\Ed$ by:
		\begin{align*}
			\iota_n^\dagger :E_n^\dagger:=\Inf^n_{e} E^\dagger\hookrightarrow E^\dagger, \\
			\iota_n :E_n^\vee:=\Inf^n_{e} \Ed\hookrightarrow \Ed.
		\end{align*}
		
	\begin{defin}
 		For $n\geq0$ define
		\begin{align*}
			\Ln^\dagger:&=(\pr_E)_* (\id_E\times \iota_n^\dagger)^*\Po^\dagger,\quad \Ln:=(\pr_E)_* (\id_E\times \iota_n)^*\Po.
		\end{align*}
		Both $\Ln$ and $\Ln^\dagger$ are locally free $\Ocal_E$-modules of finite rank equipped with canonical isomorphisms
		\begin{align*}
			\triv_e:e^*\Ln^\dagger\righteq \Ocal_{E_n^\dagger},\quad \triv_e:e^*\Ln\righteq \Ocal_{E_n^\vee}
		\end{align*}
		induced by the rigidifications of the Poincar\'{e} bundle. Furthermore, $\con{\Po^\dagger}$ induces an integrable $S$-connection $\con{\Ln^\dagger}$ on $\Ln^\dagger$. We call $\Ln^\dagger$ the \emph{$n$-th geometric logarithm sheaf}. Sometimes we will write $\Lcal^\dagger_{n,E}$ to emphasize the dependence on the elliptic curve $E/S$.
	\end{defin}
	Let us discuss some immediate properties of the geometric logarithm sheaves. The reader who is familiar with the formal properties of the abstract logarithm sheaves will immediately recognize many of the following properties: The compatibility of the Poincar\'e bundle with base change along $f:T\rightarrow S$ shows immediately that the geometric logarithm sheaves are \emph{compatible with base change}, i.\,e.
	\[
		\pr_E^*\Lcal_{n,E/S}\righteq \Lcal_{n,E_T/T}, \quad \pr_E^*\Lcal^{\dagger}_{n,E/S}\righteq \Lcal_{n,E_T/T}^{\dagger}
	\]
	where $\pr_E:E_T=E\times_S T\rightarrow E$ is the projection. By restricting form the $n$-th infinitesimal thickening to the $(n-1)$-th, we obtain \emph{transition maps}
	\[
		\Lcal_n^{\dagger}\twoheadrightarrow \Lcal_{n-1}^{\dagger},\quad \Lcal_n\twoheadrightarrow \Lcal_{n-1}.
	\]
	The decompositions $\Ocal_{\Inf^1_e E^\vee}=\Ocal_S \oplus \om_{\Ed/S}$ and $\Ocal_{\Inf^1_e E^\dagger}=\Ocal_S \oplus \Hcal $ show that the transition maps $\Lcal_1\rightarrow\Ocal_E$ and $\Lcal^\dagger_1\rightarrow \Ocal_E$ sit in \emph{short exact exact sequences}
	\[
		\begin{tikzcd}
			0 \ar[r] & \pi^*\om_{E^\vee/S} \ar[r] & \Lcal_1 \ar[r] & \Ocal_E\ar[r] & 0
		\end{tikzcd}
	\]
	and
	\[
		\begin{tikzcd}
			0 \ar[r] & \Hcal_E \ar[r] & \Lcal_1^\dagger \ar[r] & \Ocal_E\ar[r] & 0.
		\end{tikzcd}
	\]
	Since $\Po^\dagger$ is the pullback of $\Po$, we obtain natural inclusions
	\[
		\Lcal_n\hookrightarrow \Lcal^\dagger_n.
	\]
	These inclusions can be interpreted as the first non-trivial step of the \emph{Hodge filtration} of the geometric logarithm sheaf $\Lcal_n^\dagger$: The Hodge filtration on $\Hcal$ induces a descending filtration of $\Ocal_E$-modules on $\Lcal^\dagger_1$ such that all morphisms in
\[
	0\rightarrow\Hcal_E\rightarrow \Lcal^\dagger_1 \rightarrow \Ocal_E \rightarrow 0
\]
are strictly compatible with the filtration. Here, $\Ocal_E$ is considered to be concentrated in filtration step $0$. Explicitly this filtration is given as
\[
	F^{-1}\Lcal_1^\dagger=\Lcal^\dagger_1 \supseteq F^{0}\Lcal_1^\dagger= \Lcal_1\supseteq F^{1}\Lcal_1^\dagger=0.
\]
Let us write $[D]\colon E\rightarrow E$ for the isogeny given by $D$-multiplication. The dual isogeny of $[D]$ is $D$-multiplication on $E^\vee$. By the universal property of the Poincar\'e bundle over $E\times_S E^\vee$, there is a unique isomorphism
	\[
		\gamma_{\id,[D]}\colon (\id\times [D])^*\Po\righteq ([D]\times \id)^*\Po.
	\]
	Since we are working over a field of characteristic zero, the $D$-multiplication induces an isomorphism on $\Inf^n_e E^\vee$:
	\[
		\begin{tikzcd}
			\Inf^n_e E^\vee\ar[r,hook]\ar[d,"\cong"] & E^\vee\ar[d,"{[D]}"]\\
			\Inf^n_e E^\vee\ar[r,hook] & E^\vee.
		\end{tikzcd}
	\]
	Restricting $\gamma_{\id,[D]}$ along $E\times_S \Ed$ and using the above commutative diagram gives
	\[
		(\pr_E)_*(\Po^\dagger|_{E\times \Inf^n_e\Ed})\cong (\pr_E)_*\left((\id\times [D])^*\Po^\dagger|_{E\times \Inf^n_e\Ed}\right)\righteq (\pr_E)_*\left(([D]\times \id)^*\Po^\dagger|_{E\times \Inf^n_e\Ed}\right).
	\]
	Recalling $\Ln^\dagger=(\pr_E)_*(\Po^\dagger|_{E\times \Inf^n_e\Ed})$ gives an \emph{invariance under isogenies} isomorphism:
	\begin{equation}\label{eq_invariance}
		\Ln^\dagger\righteq [D]^*\Ln^\dagger.
	\end{equation}

\subsection{The Fourier--Mukai transform of Laumon}
For a smooth morphism $X\rightarrow S $ let us denote by $\Dcal_{X/S}$ the sheaf of differential operators of $X/S$. Furthermore, let us denote by $D^b_{qc}(\Dcal_{X/S})$ (resp. $D^b_{qc}(\Ocal_X)$) the derived category of bounded complexes of quasi-coherent $\Dcal_{X/S}$-modules (resp. $\Ocal_X$-modules). As usually, let $E/S$ be an elliptic curve over a smooth base $S$ over a field of characteristic zero. The Poincar\'e bundle with connection $(\Po^\dagger,\con{\Po^\dagger})$ on $E\times_S E^\dagger$ serves as kernel for Laumon's Fourier--Mukai equivalence:
\begin{thm}[{\cite[(3.2)]{laumon}}]
The functor
\[
	\Phi_{\Po^\dagger}\colon D^b_{qc}(\Ocal_{E^\dagger})\rightarrow D^b_{qc}(\Dcal_{E/S}),\quad \Fcal^\bullet\mapsto R\pr_{E,*}\left((\Po^\dagger,\con{\Po^\dagger})\otimes_{\Ocal_{E\times E^\dagger}}\pr_{E^\dagger}^*\Fcal^\bullet\right)
\]
establishes an equivalence of triangulated categories.
\end{thm}
We will restrict this derived Fourier--Mukai transform to certain complexes of unipotent objects which are concentrated in a single cohomological degree. Let us first introduce the following category:
\begin{defin}
	Let $\Jcal$ be the ideal sheaf of $\Ocal_{E^\dagger}$ defined by the unit section. Let $\Uninf$ be the full subcategory of the category of quasi-coherent $\Ocal_{E^\dagger}$-modules $\Fcal$, s.t. $\Jcal^{n+1}\Fcal=0$ and  $\Jcal^{i}\Fcal/\Jcal^{i+1}\Fcal$ is a locally free $\Ocal_S=\Ocal_{E^\dagger}/\Jcal$-module of finite rank for $i=0,...,n$.
\end{defin}
The following results appear in the work of Scheider:
\begin{lem}[{\cite[Proposition 2.2.6, Theorem 2.2.12 (i)]{rene}}]\label{lem_FMScheider}
Let us write $e_{E^\dagger}\colon S\rightarrow E^\dagger$ for the unit section of the universal vectorial extension of $E^\vee$.
\begin{enumerate}
\item For a locally free $\Ocal_S$-module $\Gcal$, we have the formula
\[
	\Phi_{\Po^\dagger}((e_{E^\dagger})_*\Gcal)=\pi^*\Gcal
\]
\item The Fourier--Mukai transform of $\Fcal\in \Uninf$ is concentrated in cohomological degree zero, i.e.
	\[
		H^i(\Phi_{\Po^\dagger}(\Fcal))=0\quad \text{for } i\neq 0.
	\]
\end{enumerate}
\end{lem}
\begin{proof} We follow closely the argument of Scheider, see {\cite[Prop. 2.2.6, Thm. 2.2.12 (i)]{rene}}:\par 
$(a)$ Base change along the Cartesian diagram
\[
\begin{tikzcd}
	E \ar[r,"\id\times e_{E^\dagger}"]\ar[d,"\pi"] & E\times_S E^\dagger\ar[d,"\pr_{E^\dagger}"] \\
	S\ar[r,"e_{E^\dagger}"] & E^\dagger
\end{tikzcd}
\]
gives a canonical isomorphism of $ \Dcal_{E\times_S E^\dagger/E^\dagger}$-modules:
\[
	\pr_{E^\dagger}^*(e_{E^\dagger})_*\Gcal\simeq (\id\times e_{E^\dagger})_*\pi^*\Gcal.
\]
The trivialization of the Poincar\'e bundle gives a $\Dcal_{E/S}$-linear isomorphism
\[
	(\id\times e_{E^\dagger})^*(\Po^\dagger,\nabla_{\dagger})\simeq (\Ocal_E,\mathrm{d}).
\]
Thus, we get an isomorphism of $\Dcal_{E\times_S E^\dagger/E^\dagger}$-modules:
\[
	\Po^\dagger\otimes_{\Ocal_{E\times_S E^\dagger}}\pr_{E^\dagger}^*(e_{E^\dagger})_*\Gcal\cong (\id\times e_{E^\dagger})_*\pi^*\Gcal.
\]
Since $\id\times e_{E^\dagger}$ is affine, it is exact and we get
\begin{align*}
	\Phi_{\Po^\dagger}((e_{E^\dagger})_*\Gcal)&=R(\pr_{E})_*\left(\Po^\dagger\otimes_{\Ocal_{E\times_S E^\dagger}}\pr_{E^\dagger}^*(e_{E^\dagger})_*\Gcal\right)\cong R(\pr_{E})_*R(\id\times e_{E^\dagger})_*\pi^*\Gcal\\
	&=R(\id_E)_*\pi^*\Gcal=\pi^*\Gcal.
\end{align*}
$(b)$ Let us prove the claim by induction on $n$. For $n=0$ we have $\Fcal=(e_{E^\dagger})_*\Gcal$ for some locally-free $\Ocal_S$-module $\Gcal$. Thus, the case $n=0$ follows from $(a)$. For $n>0$ the claim follows by induction: Indeed, the exact sequence
\[
	0\rightarrow \Jcal ^n\Fcal\rightarrow \Fcal \rightarrow \Fcal/\Jcal^n\Fcal\rightarrow 0
\]
induces a triangle
\[
	\Phi_{\Po^\dagger}(\Jcal^n\Fcal)\rightarrow \Phi_{\Po^\dagger}(\Fcal)\rightarrow \Phi_{\Po^\dagger}(\Fcal/\Jcal^n\Fcal)\rightarrow \Phi_{\Po^\dagger}(\Jcal^n\Fcal)[1]
\]
in $D^b_{qc}(\Dcal_{E/S})$. Now we conclude since $\Phi_{\Po^\dagger}(\Jcal^n\Fcal)$ and $\Phi_{\Po^\dagger}(\Fcal/\Jcal^n\Fcal)$ are concentrated in degree zero by the induction hypothesis.
\end{proof}
%The above Lemma shows, that $\Phi_{\Po^\dagger}(\Fcal)$ for $\Fcal\in \Uninf$ is concentrated in cohomological degree zero. In particular, we have the explicit formula
%\[
%g	
%\]
%For an object $\Fcal\in \Uninf$ let us define
%\[
%	\hat{\Fcal}^\dagger:=(\pr_E)_*\left( (\pr_{E^\dagger})^*\Fcal\otimes_{\Ocal_{E\times_S E^\dagger}}\Po^\dagger \right).
%\]
\begin{prop}[{\cite[Theorem 2.2.12]{rene}}]\label{prop_FMequiv}
	The functor
	\[
		\Uninf\righteq \Und{E/S}, \quad \Fcal\mapsto \hat{\Fcal}^\dagger:=H^0(\Phi_{\Po^\dagger}(\Fcal))
	\]
	is a well-defined equivalence of categories.
\end{prop}
\begin{proof} Again, we follow Scheider's proof {\cite[Theorem 2.2.12]{rene}}. Let us first prove the well-definedness by induction over $n$. For $n=0$, \Cref{lem_FMScheider} $(a)$ shows that the functor is well-defined. For $n>0$ an object $\Fcal\in \Uninf$ fits in an exact sequence
\[
	0\rightarrow \Jcal ^n\Fcal\rightarrow \Fcal \rightarrow \Fcal/\Jcal^n\Fcal\rightarrow 0
\]
with $\Jcal ^n\Fcal\in \mathrm{U}_0(\Ocal_{E^\dagger})$ and $\Fcal/\Jcal^n\Fcal\in \mathrm{U}_{n-1}(\Ocal_{E^\dagger})$. We have seen in \Cref{lem_FMScheider} $(b)$ that the Fourier--Mukai transform of each term in the above exact sequence is concentrated in degree $0$, thus we get the short exact sequence
\[
	0\rightarrow H^0(\Phi_{\Po^\dagger}(\Jcal ^n\Fcal))\rightarrow H^0(\Phi_{\Po^\dagger}(\Fcal)) \rightarrow H^0(\Phi_{\Po^\dagger}(\Fcal/\Jcal^n\Fcal))\rightarrow 0.
\]
By the induction hypothesis, $H^0(\Phi_{\Po^\dagger}(\Jcal ^n\Fcal))\in \mathrm{U}_0^\dagger(E/S)$, $H^0(\Phi_{\Po^\dagger}(\Fcal/\Jcal^n\Fcal))\in \mathrm{U}_{n-1}^\dagger(E/S)$ and we deduce $H^0(\Phi_{\Po^\dagger}(\Fcal))\in \Und{E/S}$. This proves the well-definedness.\par 
Since $\Phi_{\Po^\dagger}$ is an equivalence and all considered sheaves are concentrated in cohomological degree zero, we deduce that $\Uninf\rightarrow \Und{E/S}$ is fully faithful. The proof that $\hat{(\cdot)}^\dagger$ is essentially surjective proceeds again by induction over $n$. For $n=0$, we have $\Fcal=(e_{E^\dagger})_*\Gcal$ for some locally free $\Ocal_S$-module and the explicit formula $\hat{\Fcal}^\dagger=\pi^*\Gcal$ establishes the case $n=0$. Let $n>0$ and assume that we already know that $\hat{(\cdot )}^\dagger\colon\mathrm{U}_m(\Ocal_{E^\dagger})\rightarrow \mathrm{U}_m^\dagger(E/S)$ is essentially surjective for $m<n$. Every unipotent sheaf $\Ucal\in\Und{E/S}$ sits in an exact sequence
\[
	0\rightarrow A^1\Ucal\rightarrow \Ucal \rightarrow \pi^*\Ycal\rightarrow 0
\]
for some locally free $\Ocal_S$-module $\Ycal$. By the induction hypothesis there exists $\Fcal'\in \mathrm{U}_{n-1}(\Ocal_{E^\dagger})$ and $\Fcal''\in \mathrm{U}_{0}(\Ocal_{E^\dagger}) $ with $\widehat{(\Fcal')}^\dagger\cong A^1\Ucal$ and $\widehat{(\Fcal'')}^\dagger\cong \pi^*\Ycal$, i.e.
\[
	0\rightarrow \widehat{(\Fcal')}^\dagger\rightarrow \Ucal\rightarrow\widehat{(\Fcal'')}^\dagger\rightarrow 0.
\]
This gives us a distinguished triangle
\begin{equation}\label{eq_triangle1}
	\Phi_{\Po^\dagger}(\Fcal')\rightarrow \Ucal\rightarrow \Phi_{\Po^\dagger}(\Fcal')\rightarrow \Phi_{\Po^\dagger}(\Fcal')[1]
\end{equation}
in $D^b_{qc}(\Dcal_{E/S})$. Let us denote by $\Phi_{\Po^\dagger}^{-1}$ an quasi-inverse of the equivalence $\Phi_{\Po^\dagger}$. Applied to \eqref{eq_triangle1} this yields a triangle in $D^b_{qc}(\Ocal_{E^\dagger})$:
\[
	\Phi_{\Po^\dagger}^{-1}(\Phi_{\Po^\dagger}(\Fcal'))\rightarrow \Phi_{\Po^\dagger}^{-1}(\Ucal)\rightarrow \Phi_{\Po^\dagger}^{-1}(\Phi_{\Po^\dagger}(\Fcal'))\rightarrow \Phi_{\Po^\dagger}^{-1}(\Phi_{\Po^\dagger}(\Fcal'))[1].
\]
Since $\Phi_{\Po^\dagger}^{-1}(\Phi_{\Po^\dagger}(\Fcal'))\cong \Fcal'$ and $\Phi_{\Po^\dagger}^{-1}(\Phi_{\Po^\dagger}(\Fcal''))\cong \Fcal''$ we may replace $\Phi_{\Po^\dagger}^{-1}(\Phi_{\Po^\dagger}(\Fcal'))$ by $\Fcal'$ and $\Phi_{\Po^\dagger}^{-1}(\Phi_{\Po^\dagger}(\Fcal''))$  by $\Fcal''$ in the above triangle:
\[
	\Fcal'\rightarrow \Phi_{\Po^\dagger}^{-1}(\Ucal)\rightarrow \Fcal''\rightarrow \Fcal'[1].
\]
Because $\Fcal'$ and $\Fcal''$ are concentrated in degree zero we deduce that $\Phi_{\Po^\dagger}^{-1}(\Ucal)$ is concentrated in degree zero. Applying $H^0$ gives:
\[
	0\rightarrow \Fcal'\rightarrow H^0(\Phi_{\Po^\dagger}^{-1}(\Ucal))\rightarrow \Fcal''\rightarrow 0
\]
As $\Fcal'\in \mathrm{U}_{n-1}(\Ocal_{E^\dagger})$ and $\Fcal''\in \mathrm{U}_{0}(\Ocal_{E^\dagger}) $ we get $\Fcal:=H^0(\Phi_{\Po^\dagger}^{-1}(\Ucal))\in \Und{E/S}$. Applying $\widehat{(\cdot)}^\dagger$ gives
\[
	0\rightarrow \widehat{(\Fcal')}^\dagger\rightarrow \widehat{(\Fcal)}^\dagger\rightarrow\widehat{(\Fcal'')}^\dagger\rightarrow 0.
\]
Using once again, that $\Phi_{\Po^\dagger}^{-1}(\Ucal)$ is concentrated in degree zero, we deduce 
\[
\widehat{\Fcal}^\dagger=H^0(\Phi_{E^\dagger}(H^0(\Phi_{\Po^\dagger}^{-1}(\Ucal))))\cong \Ucal.
\]
This proves that $\Ucal$ is in the essential image of $\widehat{(\cdot)}^\dagger$ and concludes the induction step of the essential surjectivity of $\widehat{(\cdot)}^\dagger$.
\end{proof}
Finally, let us state the following result:
\begin{prop}[{\cite[Prop. 2.2.16]{rene}}]\leavevmode\label{DR_propFM}
For $\Fcal\in\Uninf$ there is a canonical isomorphism
	\[
		e^*\widehat{\Fcal}^\dagger\righteq (\pi_n)_*\Fcal.
	\]
	Here, $\pi_{n}: \Inf^n_e E^\dagger\rightarrow S$ is the structure morphism of $\Inf^n_e E^\dagger$. 
\end{prop}
\begin{proof} For the convenience of the reader, let us recall Scheider's proof. Let us view $\Fcal\in\Uninf$ as an $\Ocal_{\Inf^n_e E^\dagger}$-module. The base change isomorphism associated to the diagram
\[
\begin{tikzcd}
	E\times_S\Inf^n E^\dagger\ar[r,"\id\times\iota_n"]\ar[d,"\pr_{\Inf^n_e E^\dagger}"] & E\times_S E^\dagger \ar[d,"\pr_E"] \\
	\Inf^n E^\dagger \ar[r,"\iota_n"] & E
\end{tikzcd}
\]
gives an horizontal isomorphism 
\[
	\widehat{\Fcal}^\dagger \cong (\pr_E)_*\left(\pr_{\Inf^n E^\dagger}^*\Fcal\otimes_{\Ocal_{E\times_S \Inf^n E^\dagger}} (\id\times\iota_n)^*\Po^\dagger \right).
\]
Again, by base change along the diagram
\[
\begin{tikzcd}
	\Inf^n E^\dagger\ar[r,"\pi_n"]\ar[d,"e\times\id"] & S \ar[d,"e"] \\
	E\times_S \Inf^n E^\dagger \ar[r,"\pr_E"] & E
\end{tikzcd}
\]
we obtain the desired $\Ocal_S$-linear isomorphism
\begin{align*}
	e^*\widehat{\Fcal}^\dagger &\cong  e^*(\pr_E)_*\left(\pr_{\Inf^n E^\dagger}^*\Fcal\otimes_{\Ocal_{E\times_SE^\dagger}} (\id\times\iota_n)^*\Po^\dagger \right)\cong\\
	&\cong (\pi_n)_* (e\times\id)^*\left(\pr_{\Inf^n E^\dagger}^*\Fcal\otimes_{\Ocal_{E\times_SE^\dagger}} (\id\times\iota_n)^*\Po^\dagger \right) \cong\\
	&\cong (\pi_n)_*\Fcal.
\end{align*}
\end{proof}

\subsection{The geometric logarithm sheaves as de Rham logarithm sheaves}
The aim of this section is to show that the geometric logarithm sheaves $\Ln^\dagger$ give us a concrete geometric realization of the abstractly defined de Rham logarithm sheaves. This is one of the main results of Scheider \cite[Theorem 2.3.1]{rene}. By working with the universal property of the logarithm sheaves instead of its extension class, we can give a much simpler proof then the original one. The idea our proof is quite simple: The pair $(\Ocal_{E^\dagger}/\Jcal^{n+1},1)$ is initial in the category consisting of pairs $(\Fcal,s)$ of unipotent $\Ocal_{E^\dagger}/\Jcal^{n+1}$ modules $\Fcal\in \mathrm{U}_n(\Ocal_{E^\dagger})$ with a marked section $s\in \Gamma(S,(\pi_n)_*\Fcal)$. Thus, the Fourier--Mukai transform of $(\Ocal_{E^\dagger}/\Jcal^{n+1},1)$ should be initial in the corresponding category of unipotent vector bundles (with a marked section) through the equivalence of categories:
\[
	\Uninf\righteq \Und{E/S}.
\]
Now recall, that being initial in the category of unipotent vector bundles with a marked section is exactly the universal property of the logarithm sheaves. Finally, the formula
\[
	\widehat{\left(\Ocal_{E^\dagger}/\Jcal^{n+1}\right)}^\dagger=(\pr_E)_*\left( (\pr_{E^\dagger})^*\Ocal_{\Inf^n_e E^\dagger}\otimes_{\Ocal_{E\times E^\dagger}} \Po^\dagger \right)=(\pr_E)_*(\id\times \iota_n^\dagger)^*\Po^\dagger\stackrel{\Def}{=}\Ln^\dagger,
\]
which has already been observed by Scheider, allows us to conclude that the geometric logarithm sheaves satisfy the universal property of the abstractly defined logarithm sheaves.\par 

In order to make this argument work, we have to be more precise: First, let us observe that
\[
	e^*\Ln^\dagger\righteq \Ocal_{\Inf^n_e E^\dagger}=\Ocal_{E^\dagger}/\Jcal^{n+1}
\]
gives us a canonical section $1\in \Gamma(S,\Ocal_{\Inf^n_e E^\dagger})\cong \Gamma(S,e^*\Ln^\dagger)$ of $e^*\Ln^\dagger$. Furthermore, let us observe that the connection $\con{\Ln^\dagger}$ is a connection relative $S$, while the connection $\con{\Logn}$ is an absolute connection. In order to compare both objects, we have to restrict the connection $\con{\Logn}$ relative $S$. Let us define $\con{\Logn,E/S}:=\res_S(\con{\Logn})$. Now, we can state one of the main results of Scheider's PhD thesis:
\begin{thm}[{Scheider,\cite[Thm 2.3.1]{rene}}]\label{DR_thmRene}
There is a unique horizontal isomorphism
\[
	(\Logn,\con{\Logn,E/S})\rightarrow (\Ln^\dagger,\con{\Ln^\dagger})
\]
mapping $\idM^{(n)}$ to $1$ after pullback along $e$.
\end{thm}
The proof of  is long and involved. Scheider compares the extension classes of both tuples. It is much simpler to deduce this isomorphism by proving that both objects satisfy the same universal property, or stated differently that they are initial in the same category. Since we already know the universal property of the abstractly defined logarithm sheaves, it remains to show that the geometric logarithm sheaves satisfy the same universal property:
\begin{thm}\label{DR_thmUnivLog}
	The tuple $(\Ln^\dagger,\con{\Ln^\dagger},1)$ is the unique tuple, up to unique isomorphism, consisting of an object $(\Ln^\dagger,\con{\Ln^\dagger})\in \mathrm{U}^\dagger_n(E/S)$ and a section 
	$$1\in \Gamma(S,e^*\Ln)=\Gamma(S,\Ocal_{\Inf^n_e E^\dagger})$$
	such that the following universal property holds: For all $\Gcal\in \mathrm{U}^\dagger_n(E/S)$ the map
	\[
		\pi_* \HOM^{\mathrm{hor}}_{E/S}\left(\Ln^\dagger,\Gcal \right)\rightarrow e^*\Gcal,\quad f\mapsto (e^*f)(1)
	\]
	is an isomorphism of $\Ocal_S$-modules.
\end{thm}
\begin{proof}
	 Let $\Gcal\in \Und{E/S}$. By the equivalence
	\[
		\widehat{(\cdot)}^\dagger:\Uninf\righteq \Und{E/S},
	\]
	we may assume $\Gcal=\hat{\Fcal}^\dagger$ for some $\Fcal\in \Uninf$. Then, we have the following chain of isomorphisms
	\begin{align*}
		\pi_*\HOM_{\Und{E/S}}\left( \widehat{\Ocal_{E^\dagger}/\Jcal^{n+1}}^\dagger,\Gcal  \right) & \stackrel{(A)}{\cong} (\pi_n)_*\HOM_{\Uninf}\left( \Ocal_{E^\dagger}/\Jcal^{n+1}, \Fcal \right)= \\
		&=(\pi_n)_*\Fcal \stackrel{(B)}{\cong}e^*\hat{\Fcal}^\dagger= e^*\Gcal
	\end{align*}
	where $(A)$ is induced by the Fourier--Mukai type equivalence \Cref{prop_FMequiv} and $(B)$ is \Cref{DR_propFM}. It is straightforward to check that this chain of isomorphisms sends $f$ to $(e^*f)(1)$. This proves the universal property of $(\Ln^\dagger,\con{\Ln^\dagger},1)$.
\end{proof}

As a corollary, we can deduce the following result about absolute connections:

\begin{cor}[{\cite[Cor. 2.3.2]{rene}}]\label{DR_lemdefLn}
	There exists a unique $\KK$-connection $\con{\Ln^\dagger}^{\text{abs}}$ on $\Ln^\dagger$, such that:
\begin{enumerate}
\item $\con{\Ln^\dagger}^{\text{abs}}$ restricts to the $S$-connection $\con{\Ln^\dagger}$, i.e.~$\res_S(\con{\Ln^\dagger}^{\text{abs}})=\con{\Ln^\dagger}$ and
\item $\left(\Ln^\dagger,\con{\Ln^\dagger}^{abs}, 1 \right)$ satisfies the universal property of the absolute $n$-th de Rham logarithm sheaf stated in \Cref{log_univ}.
\end{enumerate}	
\end{cor}
\begin{proof} Uniqueness follows immediately from the universal property in $(b)$. For the existence of $\con{\Ln^\dagger}^{\text{abs}}$ let us consider $n$-th (absolute) de Rham logarithm sheaves $\left((\Logn,\con{\Logn}),\idM^{(n)} \right)$. The forgetful functor
\[
	\mathrm{res}_S\colon \Und{E/S/\KK}\rightarrow \Und{E/S}
\]
gives an object $\left(\Logn,\mathrm{res}_S(\con{\Logn}),\idM^{(n)} \right)$ in the category $\Und{E/S}$ satisfying the following universal property (by \Cref{log_univ}): For all $\Gcal\in \mathrm{U}^\dagger_n(E/S)$ the map
	\[
		\pi_* \HOM^{\mathrm{hor}}_{E/S}\left( \Logn ,\Gcal \right)\rightarrow e^*\Gcal,\quad f\mapsto (e^*f)(\idM^{(n)})
	\]
	is an isomorphism of $\Ocal_S$-modules. Since $\left( \left(\Ln^\dagger,\con{\Ln^\dagger} \right), 1 \right)$ satisfies the same universal property (see \Cref{DR_thmUnivLog}), there is a unique isomorphism
	\[
		\alpha:\Ln^\dagger \righteq \Logn
	\]
	which is $S$-horizontal, i.e.~$\alpha^*\mathrm{res}_S(\con{\Logn})=\con{\Ln^\dagger}$ and satisfies $(e^*\alpha)(1)=\idM^{(n)}$. Let us define $\con{\Ln^\dagger}^{\text{abs}}:=\alpha^* \con{\Logn}$. With this definition $\alpha$ provides an isomorphism
	\[
	(\Ln^\dagger,\con{\Ln^\dagger}^{\text{abs}},1)\righteq \left((\Logn,\con{\Logn}),\idM^{(n)} \right)
	\]
	in the category $\Und{E/S/K}$. In particular, the tuple $(\Ln^\dagger,\con{\Ln^\dagger}^{\text{abs}},1)$ satisfies the universal property of the $n$-th de Rham logarithm sheaf, i.e.~$(b)$ holds. $(a)$ follows from the formula
	\[
		\mathrm{res}_S(\con{\Ln^\dagger}^{\text{abs}})=\mathrm{res}_S(\alpha^*\con{\Logn})=\alpha^*\mathrm{res}_S(\con{\Logn})=\con{\Ln^\dagger}.
	\]
\end{proof}

\subsection{Extension classes and the Kodaira--Spencer map}
Simultaneously with the geometric logarithm sheaves $(\Ln^\dagger,\con{\Ln^\dagger})$, we introduced a variant $\Ln$ of geometric logarithm sheaves without a connection. One might ask about a similar universal property for $\Ln$. Indeed, let us define $U_n(E/S)$ as the full subcategory of the category of vector bundles consisting of unipotent objects of length $n$ for $E/S$, i.e. there is a $n$-step filtration with graded pieces of the form $\pi^*\Gcal$ for some vector bundle $\Gcal$ on $S$. The pullback $e^*\Ln=\Ocal_{\Inf^n_e E^\vee}$ is equipped with a distinguished section $1\in\Gamma(S,\Ocal_{\Inf^n_e E^\vee})$.

\begin{thm}\label{DR_thmExtLcal}
	The pair $(\Ln,1)$ is the unique pair, up to unique isomorphism, consisting of a unipotent vector bundle of length $n$ for $E/S$ and a section 
	$$1\in \Gamma(S,e^*\Ln)=\Gamma(S,\Ocal_{\Inf^n_e E^\vee})$$
	such that the following universal property holds: For all $\Ucal\in \mathrm{U}_n(E/S)$ the map
	\[
		\pi_* \HOM_{\Ocal_E}\left(\Ln,\Gcal \right)\rightarrow e^*\Gcal,\quad f\mapsto (e^*f)(1)
	\]
	is an isomorphism of $\Ocal_S$-modules.
\end{thm}
\begin{proof}
	The same proof as in \Cref{DR_thmUnivLog} works if one replaces the Fourier-Mukai transform of Laumon by the classical Fourier-Mukai transform.
\end{proof}
Let us recall that $\Lcal^1$ sits in a short exact sequence
		\begin{equation*}
			\begin{tikzcd}		
				\quad 0 \ar[r] & \pi^* \om_{\Ed/S} \ar[r] & \Lcal^1 \ar[r] & \Ocal_E \ar[r] & 0
			\end{tikzcd}		
		\end{equation*}
		and that the rigidification of the Poincar\'e bundle induces a trivializing isomorphism along the zero section $\triv_e\colon e^*\Lcal_1\righteq \Ocal_S\oplus \om_{\Ed/S}$. An immediate reformulation of the above universal property of $\Lcal^1$ is the following:
	\begin{cor}\label{GL_cor_pushout}
		Let $E/S$ be an elliptic curve, $M$ a locally free $\Ocal_S$-module of finite rank and let $(\Fcal,\sigma)$ be a pair consisting of an extension
		\begin{equation*}
			\begin{tikzcd}		
				\Fcal:\quad 0 \ar[r] & \pi^* M \ar[r] & F \ar[r] & \Ocal_E \ar[r] & 0
			\end{tikzcd}		
		\end{equation*}
		together with a splitting of $e^*\Fcal$, i.\,e. $\sigma$ is an isomorphism $e^*F\righteq \Ocal_S\oplus M$ which is compatible with the extension structure. Then, there is a unique morphism 
		\[
			\varphi: \om_{\Ed/S}\rightarrow M
		\]
		such that the pair $(\Fcal,\sigma)$ is the pushout of the pair $(\Lcal_1,\triv_e)$ along $\varphi$.
	\end{cor}
	\begin{proof}
		The pair $(\Fcal,\sigma)$ induces a pair $(F,s)$ with $F\in U_1(E/S)$ and $s:=\sigma^{-1}(1,0)\in\Gamma(S,e^*F)$. By the universal property of $\Lcal^1$, there is a unique morphism $f:\Lcal_1\rightarrow F$ which identifies $1\in \Gamma(S,e^*\Lcal_1)$ with $s\in \Gamma(S,e^*F)$. The pushout of
		\begin{equation*}
			\begin{tikzcd}		
				\Fcal:\quad 0 \ar[r] & \pi^* M \ar[r] & F \ar[r] & \Ocal_E \ar[r] & 0
			\end{tikzcd}		
		\end{equation*}		
		along $\varphi:=(e^*f)|_{\om_{E^\vee/S}}:\om_{\Ed/S}\rightarrow M$ is isomorphic to $(\Fcal,\sigma)$. Uniqueness follows from the rigidity of extensions with a fixed splitting along the pullback $e^*$.
	\end{proof}

An interesting application of this corollary relates $\Lcal_1$ to the absolute K\"ahler differentials of the universal elliptic curve. Let $E\rightarrow S$ be an elliptic curve and $S\rightarrow T$ be a smooth morphism. We have the following fundamental short exact sequences of K\"ahler differentials:
	\begin{equation}\label{KS_eq2}
	\begin{tikzcd}
		0 \ar[r] & \pi^* \Omega^1_{S/T} \ar[r] & \Omega^1_{E/T} \ar[r] & \Omega^1_{E/S} \ar[r] & 0
	\end{tikzcd}
	\end{equation}
	and
	\begin{equation*}
	\begin{tikzcd}
		0 \ar[r] & \Ical/\Ical^2 \ar[r] & e^*\Omega^1_{E/T} \ar[r] & \Omega^1_{S/T} \ar[r] & 0
	\end{tikzcd}
	\end{equation*}
	with $\Ical$ the ideal sheaf defining the zero section $e:S\rightarrow E$. The second short exact sequence induces a canonical splitting
	\[
		e^*\Omega^1_{E/T} \righteq \Omega^1_{S/T}\oplus e^*\Omega^1_{E/S}
	\]
	of the pullback of \eqref{KS_eq2} along $e$. The pullback of the Kodaira--Spencer map 
	\[
		\om_{E/S}\otimes \om_{\Ed/S}\rightarrow \Omega^1_{S/T}
	\]
	along $\pi:E\rightarrow S$ induces a map
	\begin{equation*}
		\mathrm{ks}: \Omega^1_{E/S}\otimes_{\Ocal_E} \pi^*\om_{\Ed/S}\rightarrow \pi^* \Omega^1_{S/T}.
	\end{equation*}
	\begin{cor}
		The short exact sequence of K\"ahler differentials:
	\begin{equation*}
	\begin{tikzcd}
		0 \ar[r] & \pi^* \Omega^1_{S/T} \ar[r] & \Omega^1_{E/T} \ar[r] & \Omega^1_{E/S} \ar[r] & 0
	\end{tikzcd}
	\end{equation*}
		is the pushout of $\Lcal^1\otimes_{\Ocal_E}\Omega^1_{E/S}$ along the Kodaira--Spencer map $\mathrm{ks}$.
	\end{cor}
	\begin{proof}
	Assume, we have a short exact sequence 
			\begin{equation*}
			\begin{tikzcd}		
				\Fcal:\quad 0 \ar[r] & \pi^* M \ar[r] & F \ar[r] & \Ocal_E \ar[r] & 0
			\end{tikzcd}		
		\end{equation*}
		together with a splitting $\sigma$ of $e^*\Fcal$ as in \Cref{GL_cor_pushout}. Then we can describe the unique map
		\[
			\varphi: \om_{\Ed/S}\rightarrow M
		\]
		whose existence is guaranteed by \Cref{GL_cor_pushout}, as the image of $1$ under the connecting morphism
		\[
			\delta: \Ocal_E\rightarrow R^1\pi_*(\pi^* M)=\Hom_{\Ocal_S}(\om_{\Ed/S},M).
		\]
		Now the result follows from the definition of the Kodaira--Spencer map \cite[p. 80]{chai_faltings}: The Kodaira--Spencer map is the image of $1$ under the connecting homomorphism of the short exact sequence
		\begin{equation}\label{ses_1}
		\begin{tikzcd}
			0 \ar[r] & \pi^* \Omega^1_{S/T}\otimes_{\Ocal_E} \left( \Omega^1_{E/S}\right)^\vee \ar[r] & \Omega^1_{E/T}\otimes_{\Ocal_E} \left( \Omega^1_{E/S}\right)^\vee \ar[r] & \Ocal_E \ar[r] & 0
		\end{tikzcd}			
		\end{equation}
		obtained by tensoring the short exact sequence of K\"ahler differentials with $\otimes_{\Ocal_E} \left( \Omega^1_{E/S}\right)^\vee$.
	\end{proof}
	In particular, the pushout along the Kodaira--Spencer map induces a map
	\begin{equation}\label{eq_KS}
		\KS\colon \Lcal_1\otimes_{\Ocal_E}\Omega^1_{E/S}\rightarrow \Omega^1_{E/T}.
	\end{equation}
	If $E/S$ is the universal elliptic curve over the modular curve (with some level structure) then $\KS$ is an isomorphism. In particular, $\Lcal_1$ is essentially given by the absolute K\"ahler differentials.
\section{The de Rham realization of the elliptic polylogarithm}
The aim of this section is to give an explicit algebraic description of the de Rham realization of the elliptic polylogarithm in terms of the Kronecker section of the Poincar\'{e} bundle for arbitrary families of elliptic curves $E/S$ over a smooth separated $\KK$-scheme $S$ of finite type. From now on, we will use Scheider's explicit description of the de Rham logarithm sheaves and fix $\left(\Ln^\dagger,\con{\Ln^\dagger}^{abs}, 1 \right)$ as an explicit model for the de Rham logarithm sheaves (cf.~\Cref{DR_lemdefLn}).\par

\subsection{The Kronecker section of the geometric logarithm sheaves}
Recall that the dual elliptic curve $E^\vee$ represents the connected component of the functor
\[
	T\mapsto \mathrm{Pic}(E_T/T):=\{ \text{iso. classes of rigidified line bundles } (\Lcal,r) \text{ on } E_T/T\}
\]
on the category of $S$-schemes. The polarization associated to the ample line bundle $\Ocal_E([e])$ gives us an explicit autoduality isomorphism
\begin{equation}\label{EP_eqAutodual}
\begin{tikzcd}[row sep=tiny]
	\lambda: E \ar[r] & \Pic=:\Ed\\
	P\ar[r,mapsto] & \left( \Ocal_E([-P]-[e])\otimes_{\Ocal_E} \pi^*e^*\Ocal_E([-P]-[e])^{-1}, \mathrm{can} \right)
\end{tikzcd}
\end{equation}
Here, $\mathrm{can}$ is the canonical rigidification given by the canonical isomorphism
\[
	e^*\Ocal_E([-P]-[e])\otimes_{\Ocal_S} e^*\Ocal_E([-P]-[e])^{-1} \righteq \Ocal_S.
\]
With this chosen identification of $E$ and $E^\vee$ we can describe the bi-rigidified Poincar\'e bundle as follows
\begin{align*}
	(\Po,r_0,s_0):&=\left( \pr_1^* \Ocal_E([e])^{\otimes-1}\otimes\pr_2^*\Ocal_{E}([e])^{\otimes-1}\otimes \mu^*\Ocal_E([e])\otimes \pi_{E\times E}^* \om_{E/S}^{\otimes -1} ,r_0,s_0 \right).
\end{align*}
Here, $\Delta=\ker \left(\mu:E\times \Ed\rightarrow E\right)$ is the anti-diagonal and $r_0,s_0$ are the rigifications induced by the canonical isomorphism 
\[
	e^*\Ocal_E(-[e])\righteq \om_{E/S}.
\]
This description of the Poincar\'{e} bundle gives the following isomorphisms of locally free $\Ocal_{E\times E}$-modules, i.\,e. all tensor products over $\Ocal_{E\times E}$:
\begin{align}\label{ch_EP_eq5}
	\notag \Po\otimes \Po^{\otimes -1} &= \Po \otimes \left( \Ocal_{E\times E}(-[e\times E]-[E\times e] + \Delta)\otimes \pi_{E\times E}^* \omega_{E/S}^{\otimes-1}\right)^{\otimes-1}\cong \\
	 &\cong \Po \otimes \Omega^1_{E\times E/E}([e\times E]+[E\times e]) \otimes \Ocal_{E\times E}( - \Delta)
\end{align}
The line bundle $\Ocal_{E\times E}( - \Delta)$ can be identified with the ideal sheaf $\Jcal_\Delta$ of the anti-diagonal $\Delta$ in $E\times_S E$ in a canonical way. If we combine the inclusion 
\[
	\Ocal_{E\times E}( - \Delta)\cong \Jcal_\Delta\hookrightarrow \Ocal_{E\times E}
\]
with \eqref{ch_EP_eq5}, we get a morphism of $\Ocal_{E\times E}$-modules
\begin{equation}\label{ch_EP_eq6}
	\Po\otimes \Po^{\otimes -1} \hookrightarrow \Po \otimes \Omega^1_{E\times E/E}([e\times E]+[E\times e]).
\end{equation}
Recall, that we agreed to identify $E$ with its dual with the canonical polarization associated to the line bundle $\Ocal_E([e])$. With this identification, the \emph{Kronecker section} 
	\[
		s_{\mathrm{can}}\in \Gamma\left(E\times_S E^\vee, \Po \otimes_{\Ocal_{E\times E^\vee}} \Omega^1_{E\times E^\vee/E^\vee}([e\times E^\vee]+[E\times e])\right)
	\]
is then defined as the image of the identity section $\id_{\Po}\in\Gamma(E\times E, \Po\otimes \Po^{\otimes -1})$ under \eqref{ch_EP_eq6}. The universal property of the Poincar\'e bundle gives us a canonical isomorphisms for $D>1$:
\[
	\gamma_{1,D}\colon (\id\times [D])^*\Pcal\righteq ([D]\times \id)^*\Pcal
\]
Let us define the \emph{$D$-variant of the Kronecker section} by
\[
			\scan^D:= D^2\cdot  \gamma_{1,D}\left( (\id \times [D])^*(\scan)\right)-([D]\times \id)^*(\scan).
\]
This is a priori an element in
\[
		\Gamma\left( E\times_S \Ed, ([D]\times \id)^*\left[ \Po\otimes \Omega^1_{E\times \Ed/\Ed}\left( [E\times \Ed[D]]+ [E[D]\times E^\vee]\right) \right] \right),
\]
but below, we will see that it is actually contained in
\[
			\Gamma\left( E\times_S \Ed, ([D]\times \id)^*\left[ \Po\otimes \Omega^1_{E\times \Ed/\Ed}\left( [E\times (\Ed[D]\setminus \{e\})]+ [E[D]\times E^\vee]		\right) \right] \right).
	\]
In other words, passing from the Kronecker section to its $D$-variant removes a pole along the divisor $E\times e$. For the proof, we will need the distribution relation of the Kronecker section: For a torsion section $t\in E[D](S)$, let us define the translation operator
\[
	\Ucal_t^D:= \gamma_{1,D}\circ(\id\times T_t)^*\gamma_{1,D}^{-1}\colon ([D]\times T_t)^*\Po\rightarrow ([D]\times\id)^*\Po.
\]
For a section $f\in \Gamma(E\times E^\vee,\Po\otimes_{\Ocal_{E\times E^\vee/E^\vee}}\Omega^1_{E\times E^\vee/E^\vee}(E\times e + e\times E]))$ let us introduce the notation
\[
	U_t^D(f):=(\Ucal_t^D\otimes \id_{\Omega^1})(([D]\times T_t)^*(f)).
\]
With this notation, the distribution relation of the Kronecker section reads in the convenient form:
\begin{prop}[Distribution Relation]
	For an elliptic curve $E/S$ with $D$ invertible on $S$ and $|E[D](S)|=D^2$ we have
	\[
		\sum_{e\neq t\in E[D](S)}U_t^D(\scan)=\scan^D.
	\]
\end{prop}
\begin{proof}Let us refer to {\cite[Corollary A.3]{EisensteinPoincare}} for the proof of the distribution relation.
\end{proof}
The distribution relation shows, that passing to the $D$-variant of the Kronecker section removes a pole along the divisor $E\times e$:
\begin{cor}
	The $D$-variant of the Kronecker section is contained in
	\[
			\Gamma\left( E\times_S \Ed, ([D]\times \id)^*\left[ \Po\otimes \Omega^1_{E\times \Ed/\Ed}\left( [E\times (\Ed[D]\setminus \{e\})]+ [E[D]\times E^\vee]		\right) \right] \right).
	\]
\end{cor}
\begin{proof}
By definition of the translation operators $U_t^D$, we get
\[
	U^D_t(s_{\mathrm{can}})\in \Gamma\left(E \times_S \Ed, ([D]\times\id)^*\left(\Po\otimes\Omega^1_{E \times_S \Ed/\Ed}\left([E\times{(-t)}]+[{e}\times \Ed] \right)\right)\right).
\]
Now the result follows from the distribution relation.
\end{proof}

%\begin{rem}%TODO change
%Let us recall, that sections of line bundles of complex abelian varieties can be described in terms of theta functions.	The terminology \emph{Kronecker section} comes from the fact, that for a complex elliptic curve, one can use the classical Kronecker theta function to describe the section explicitly. Indeed, it was a deep insight of Bannai, Kobayashi and Tsuji that the Kronecker theta function is closely related to the elliptic polylogarithm, \cite{BKT}.
%\end{rem}
The rigidification $(\id\times e)^*\Po\cong \Ocal_E$ of the Poincar\'e bundle gives us the identification:
\[
	([D]\times e)^*\left(\Po\otimes\Omega^1_{E \times_S \Ed/\Ed}\left([E\times{(-t)}]+[{e}\times \Ed] \right)\right)\cong [D]^*\Omega^1_{E/S}([e])\cong\Omega^1_{E/S}(E[D]).
\]
Using this identification, we get $(\id\times e)^*\scan\in \Gamma(E,\Omega^1_{E/S}(E[D]))$. This allows us to identify the $1$-form $(\id\times e)^*\scan$ with the logarithmic derivative of the Kato--Siegel function, which has been introduced by Kato in \cite[Proposition 1.3]{kato}. Let us briefly recall the definition of the Kato--Siegel function:
\begin{prop}[{\cite[Proposition 1.3]{kato}}]
	Let $E$ be an elliptic curve over a base scheme $S$ and $D$ be an integer which is prime to $6$. There exists a unique section $\thetaD\in \Gamma(E\setminus E[D],\Ocal_E^\times)$ satisfying the following conditions:
	\begin{enumerate}
	\item $\thetaD$ has divisor $D^2[e]-[E[D]]$,
	\item $\mathrm{N}_M(\thetaD)=\thetaD$ for every integer $M$ which is prime to $D$. Here $\mathrm{N}_M\colon \Gamma(E\setminus E[MD],\Ocal_E)\rightarrow \Gamma(E\setminus E[D],\Ocal_E)$ denotes the norm map along $M$.
	\end{enumerate}
	The function $\thetaD$ is called \emph{Kato--Siegel} function.
\end{prop}
The Kato--Siegel functions play an important role in modern number theory. While their values at torsion points are elliptic units, the logarithmic derivatives $d \log \thetaD\in\Gamma(E,\Ocal_E(E[D]))$ are closely related to algebraic Eisenstein series. These two properties make them an important tool for proving explicit reciprocity laws for twists of Tate modules of elliptic curves, see for example \cite{kato_IwasawaBdR} and \cite{tsuji}. Let us remark, that the logarithmic derivatives of the Kato--Siegel functions appear as specializations of the Kronecker section:
\begin{prop}[{\cite[Cor. 5.7]{EisensteinPoincare}}]\label{prop_KatoSiegel}
	The section
	\[
		(\id\times e)^*\scan^D \in \Gamma(E,\Omega^1_{E/S}(E[D]))
	\]
	coincides with the logarithmic derivative of the Kato--Siegel function, i.e.
	\[
		(\id\times e)^*\scan^D=d\log \thetaD.
	\]	
\end{prop}
\begin{proof}
In {\cite[Cor. 5.7]{EisensteinPoincare}} we have proven
\[
	d\log \thetaD= \sum_{e\neq t\in E[D](S)}(\id\times e)^*(U_t^D(\scan)).
\]
Now, the result follows from the distribution relation.
\end{proof}
By restricting $\scan^D$ along $E\times \Inf^n_e E^\vee$ we obtain a compatible system of $1$-forms with values in the logarithm sheaves, more precisely: The canonical isomorphism $[D]^*\Omega^1_{E/S}([e])\righteq \Omega^1_{E/S}(E[D])$ tensored with the invariance under isogenies isomorphism (cf. \eqref{eq_invariance}) gives
		\begin{equation}\label{GL_eq8}
			[D]^*\left(\Ln \otimes_{\Ocal_E} \Omega^1_{E/S}\left( [e] \right)\right)\righteq \Ln\otimes_{\Ocal_E} \Omega^1_{E/S}\left( E[D] \right).
		\end{equation}
\begin{defin}
	Define the \emph{$n$-th infinitesimal Kronecker section}
		\[
			\lnD\in  \Gamma\left(E,\Ln \otimes_{\Ocal_E} \Omega^1_{E/S}\left( E[D] \right)\right)
		\]
		as the image of $(\pr_E)_*(\id\times \iota_n)^*(\scan^D)\in \Gamma\left(E,[D]^*\left[\Ln \otimes_{\Ocal_E} \Omega^1_{E/S}\left( [e] \right)\right]\right)$ under \eqref{GL_eq8}.
\end{defin}
By construction the $\lnD$ form a compatible system of sections with respect to the transition maps of the geometric logarithm sheaves. By abuse of notation, we will denote the image of $\lnD$ under the canonical inclusion
\[
	\Lcal_n \hookrightarrow \Lcal^\dagger_n 
\]
again by $\lnD$.

\subsection{Lifting the infinitesimal Kronecker sections}
By the results of Scheider, we already have a good explicit model of the de Rham logarithm sheaves in terms of the Poincar\'e bundle. Our next aim is to give an explicit description of the elliptic de Rham polylogarithm in terms of the Poincar\'e bundle. Therefore, we would like to lift the infinitesimal Kronecker sections
\[
	\lnD\in \Gamma\left( E, \Ln\otimes_{\Ocal_E} \Omega^1_{E/S}(E[D]) \right)\subseteq \Gamma\left( E, \Ln^\dagger\otimes_{\Ocal_E} \Omega^1_{E/S}(E[D]) \right)
\]
to absolute $1$-forms. In \cref{eq_KS} we have used the Kodaira-Spence map to construct a map
\[
	\KS\colon\Lcal^1\otimes_{\Ocal_E}\Omega^1_{E/S}\rightarrow \Omega^1_{E/\KK}.
\]
Further, by the universal property of $\Lcal_n$ we have comultiplication maps $\Lcal_{k+l}\rightarrow \Lcal_k\otimes \Lcal_l$ mapping $1$ to $1\otimes 1$ after $e^*$. If we combine the Kodaira--Spencer map with the comultiplication of the geometric logarithm sheaves, we obtain a map
\begin{equation}\label{eq_liftingmap}
	\Lcal_{n+1}\otimes\Omega^1_{E/S}\rightarrow \Lcal_{n}\otimes\Lcal_1\otimes\Omega^1_{E/S} \rightarrow \Lcal_{n}\otimes \Omega^1_{E/\KK}.
\end{equation}
lifting relative $1$-forms to absolute $1$-forms.
\begin{defin}\label{DR_Pol_defLnD}
	Let us define the $n$-th \emph{absolute infinitesimal Kronecker section} $$\LnD\in\Gamma(E,\Ln\otimes\Omega^1_{E/S}(E[D]))$$ as the image of $l_{n+1}^D$ under the lifting map \eqref{eq_liftingmap}.
\end{defin}
A first step in proving that the absolute infinitesimal Kronecker sections represent the polylogarithm is the following:
\begin{prop}\label{DR_liftingLnD}
	Let us view $\LnD$ as section of $\Ln^\dagger|_{U_D}\otimes_{\Ocal_{U_D}}\Omega^1_{U_D/\KK}$ via the inclusion $\Ln\hookrightarrow \Ln^\dagger$. Then
	\[
		\LnD\in \Gamma\left( U_D, \ker\left( \Ln^\dagger\otimes_{\Ocal_{E}}\Omega^1_{E/\KK}\stackrel{ \dd^{(1)} }{\longrightarrow} \Ln^\dagger\otimes_{\Ocal_{E}}\Omega^2_{E/\KK} \right) \right)
	\]
	where $ \dd^{(1)}$ is the second differential in the absolute de Rham complex of $(\Ln^\dagger,\con{\Ln^\dagger,abs})$.
\end{prop}
\begin{proof}
The question is \'{e}tale locally on the base. Indeed, for a Cartesian diagram
	\[
		\begin{tikzcd}
		E_T\ar[r,"\tilde{f}"] \ar[d] & E\ar[d]\\
		T\ar[r,"f"] & S
		\end{tikzcd}
	\]
	with $f$ finite \'{e}tale we have an isomorphism 
	\[
		\tilde{f}^*\left( \Ln\otimes \Omega^1_{E/\KK} \right)\righteq\Lcal_{n,E_T}\otimes  \Omega^1_{E_T/\KK}
	\]
	which identifies $(\tilde{f}^*\LnD)_{n\geq 0}$ with $(\LnD)_{n\geq 0}$. Thus, we may prove the claim after a finite \'{e}tale base change. Now, choose an arbitrary $N>3$. Since we are working over a scheme of characteristic zero, the integer $N$ is invertible and there exists \'{e}tale locally a $\Gamma_1(N)$-level structure. Again, by compatibility with base change it is enough to prove the claim for the universal elliptic curve $\Ecal$  with $\Gamma_1(N)$-level structure over the modular curve $\Mcal$ over $K$. The vanishing of $\dd^{(1)}(\LnD)$ can be checked after analytification. The necessary analytic computation is shifted to \cref{sec_proof}.
\end{proof}

\subsection{The polylogarithm class via the Poincar\'{e} bundle}
The edge morphism $E^{1,0}_2\rightarrow E^1$ in the Hodge-to-de-Rham spectral sequence 
\[
	E_1^{p,q}=H^q(U_D,\Ln^\dagger\otimes_{\Ocal_{U_D}}\Omega^p_{U_D/\KK}) \Rightarrow E^{p+q}=\HdRabs{p+q}{U_D/\KK,\Ln^\dagger}
\]
induces a morphism
\begin{equation}\label{DR_Pol_edge}
	\Gamma\left(U_D,\ker\left(\Ln^\dagger\otimes_{\Ocal_{U_D}}\Omega^1_{U_D/\KK}\xrightarrow[]{\dd^{(1)}} \Ln^\dagger\otimes_{\Ocal_{U_D}}\Omega^2_{U_D/\KK}\right) \right)\xrightarrow[]{[\cdot]} \HdRabs{1}{U_D/\KK,\Ln^\dagger}.
\end{equation}
 We show that $\polD$ is represented by the compatible system $(\LnD)_{n\geq 0}$ under \eqref{DR_Pol_edge}. 
\begin{thm}\label{DR_Pol_thm}
	The $D$-variant of the elliptic polylogarithm in de Rham cohomology is explicitly given by
	\[
		\polD=([\LnD])_{n\geq0}
	\]
	where $[\LnD]$ is the de Rham cohomology class associated with $\LnD$ via \eqref{DR_Pol_edge}. Here, $\LnD$ are the (absolute) infinitesimal Kronecker sections associated to the Kronecker section.
\end{thm}
\begin{proof}
First let us recall that $\LnD$ is contained in the kernel of the first differential of the de Rham complex $\Omega^\bullet_{E/\KK}\otimes\Ln^\dagger$, c.f. \Cref{DR_liftingLnD}. Thus, $[\LnD]$ is well-defined. Further, the question is \'{e}tale locally on the base. Indeed, for a Cartesian diagram
	\[
		\begin{tikzcd}
		E_T\ar[r,"\tilde{f}"] \ar[d] & E\ar[d]\\
		T\ar[r,"f"] & S
		\end{tikzcd}
	\]
	with $f$ finite \'{e}tale we have an isomorphism 
	\[
		\tilde{f}^*\left( \Ln^\dagger\otimes \Omega^1_{E/\KK} \right)\righteq\Lcal^\dagger_{n,E_T}\otimes  \Omega^1_{E_T/\KK}
	\]
	which identifies $(\tilde{f}^*\LnD)_{n\geq 0}$ with $(\LnD)_{n\geq 0}$. Furthermore, the canonical map
	\[
		\HdR{1}{U_D,\Ln^\dagger} \righteq \HdR{1}{U_D\times_S T,\Lcal^\dagger_{n,E_T}}
	\]
	is an isomorphism and identifies the polylogarithm classes. Thus, we may prove the claim after a finite \'{e}tale base change. Now, choose an arbitrary $N>3$. Since we are working over a scheme of characteristic zero, the integer $N$ is invertible and there exists \'{e}tale locally a $\Gamma_1(N)$-level structure. Again, by compatibility with base change it is enough to prove the claim for the universal elliptic curve $\Ecal$  with $\Gamma_1(N)$-level structure over the modular curve $\Mcal$. By the defining property of the polylogarithm we have to show
	\[
		\Res\left( ([\LnD]) _{n\geq0}\right)=D^21_{e}-1_{\Ecal[D]}.
	\]
	We split this into two parts:
	\begin{enumerate}
	\item[(A)] $\Res L_0^D= D^21_{e}-1_{\Ecal[D]}$
	\item[(B)] The image of $\Res\left( ([\LnD]) _{n\geq0}\right)$ under
	\[
		\prod_{n=0}^{\infty} \HdRabs{0}{\Ecal[D],\Sym^k\Hcal_{E[D]}}\twoheadrightarrow \prod_{n=1}^{\infty} \HdRabs{0}{\Ecal[D],\Sym^k\Hcal_{E[D]}}
	\]
	is zero.
	\end{enumerate}
	(A): Since $\Mcal$ is affine, the Leray spectral sequence for de Rham cohomology shows that we obtain the localization sequence for $n=1$ by applying $\HdRabs{0}{\Mcal,\cdot}$ to
	\begin{equation*}
	\begin{tikzcd}
		0\ar[r] & \HdR{1}{\Ecal/\Mcal} \ar[r] & \HdR{1}{U_D/\Mcal}\ar[r,"\Res"] & \HdR{0}{\Ecal/\Mcal}.
	\end{tikzcd}
	\end{equation*}
	This exact sequence can be obtained by applying $R\pi_*$ to the short exact sequence
	\[
		\begin{tikzcd}
			0 \ar[r] & \Omega^\bullet_{\Ecal/\Mcal} \ar[r] & \Omega^\bullet_{\Ecal/\Mcal}(\Ecal{[D]}) \ar[r,"\Res"] & (i_{\Ecal[D]})_*\Ocal_{\Ecal[D]}{[-1]}\ar[r] & 0
		\end{tikzcd}
	\]
	of complexes. Thus, it is enough to show $\Res (l_0^D)=D^21_{e}-1_{\Ecal[D]}$. But we have already seen that $l_0^D$ coincides with the logarithmic derivative of the Kato--Siegel function, c.f.~\Cref{prop_KatoSiegel}. The residue condition $\Res (l_0^D)=D^21_{e}-1_{\Ecal[D]}$ follows immediately by one of the defining properties of the Kato--Siegel function. This proves (A).\par
	(B): The following lemma implies the vanishing of
	\[
		\HdRabs{0}{\Ecal[D],\Sym^k\Hcal_{\Ecal[D]}}=\HdRabs{0}{\Mcal,\mathrm{Sym}^k \Hcal}=0,\quad \text{for }k>0.
	\]
	So (B) holds in the universal case for trivial reasons.
\end{proof}

\begin{lem}\label{DR_Log_lemSymHdR}
	Let $N>3$ and $\Ecal/\Mcal$ be the universal elliptic curve with $\Gamma_1(N)$-level structure over $\KK$. Then
	\[
		\HdRabs{0}{\Mcal,\Sym^k \Hcal}=0
	\]
	for all $k\geq 1$.
\end{lem}
\begin{proof}
	One can show the vanishing of $\HdRabs{0}{\Mcal,\Sym^k \Hcal}$ after analytification. Then, the statement boils down, using the Riemann--Hilbert correspondence, to the obvious vanishing result
	\[
		H^0(\Gamma_1(N),\Sym^k\ZZ^2)=0,\quad k\geq 1
	\]
	in group cohomology. Here, $\ZZ^2$ is the regular representation of $\Gamma_1(N)\subseteq \mathrm{Sl}_2(\ZZ)$ on $\ZZ^2$.
\end{proof}

\begin{rem}
The above Theorem gives an explicit representative $(\LnD)_{n\geq 0}$ of the de Rham polylogarithm class. One could ask about uniqueness of this representative. We have already seen that the compatible system of sections $(\LnD)_{n\geq 0}$ is contained in the first non-trivial step
\[
	F^0\Ln^\dagger\otimes\Omega^1_{E/K}=\Ln\otimes\Omega^1_{E/K}
\]
 of the Hodge filtration. With some more effort one can actually prove that the system $(\LnD)_{n\geq 0}$ is the unique system in $F^0\Ln^\dagger\otimes\Omega^1_{E/K}$ representing the de Rham polylogarithm. Let us refer the interested reader to our PhD thesis \cite[Proposition 5.2.12]{PhD}.

\end{rem}

\subsection{Proof of \Cref{DR_liftingLnD} via the mixed heat equation for Theta functions}\label{sec_proof}
In this subsection we will pass to the universal elliptic curve and deduce \Cref{DR_liftingLnD} from the mixed heat equation
\[
		2\pi i \cdot \partial_\tau J(z,w,\tau)=\partial_z\partial_w J(z,w,\tau).
\]
of the Jacobi theta function. Let $\Ecal/\Mcal$ be the universal elliptic curve over $\mathbb{Q}$ with $\Gamma_1(N)$-level structure. The complex manifolds $\Ecal(\CC)$ and $\Mcal(\CC)$ can be explicitly described as
\[
	\Ecal(\CC)=\CC\times\HH/\ZZ^2\rtimes \Gamma_1(N)\twoheadrightarrow \Mcal(\CC)=\HH/\Gamma_1(N).
\]
Recall that we fixed the polarization associated with $\Ocal([e])$ as autoduality isomorphism. The above explicit analytification together with this autoduality isomorphism gives
\[
	\CC\times\CC\times\HH\twoheadrightarrow \Ecal(\CC)\times_{\Mcal(\CC)}\Ecal^\vee(\CC)
\]
as universal covering. Let us write $(z,w,\tau)$ for the coordinates on the universal covering. Using the autoduality isomorphism from \cref{EP_eqAutodual} we can write the rigidified Poincar\'e bundle on $\Ecal\times_\Mcal\Ecal$ as
\begin{align*}
	\Po= \pr_1^* \Ocal_\Ecal([e])^{\otimes-1}\otimes\pr_2^*\Ocal_{\Ecal}([e])^{\otimes-1}\otimes \mu^*\Ocal_\Ecal([e])\otimes \pi_{\Ecal\times \Ecal}^* \om_{\Ecal/\Mcal}^{\otimes -1}.
\end{align*}
Let us write $\tilde{\Po}$ for the pullback of the analytified Poincar\'e bundle to the above universal covering. Then $\tilde{\Po}$ is a trivial line bundle on the complex manifold $\CC\times\CC\times\HH$ and by the above explicit description of $\Po$ in terms of $\Ocal_\Ecal([e])$ it can be trivialized in terms of a suitable theta function. Let us choose the Jacobi theta function
\[
	J(z,w,\tau):=\frac{\vartheta(z+w)}{\vartheta(z)\vartheta(w)}
\]
with
\[
	\vartheta(z):=\exp\left(\frac{z^2}{2}\eta_1(\tau) \right)\sigma(z,\tau)
\]
as trivializing section of the pullback $\tilde{\Po}$ of the Poincar\'e bundle to the universal covering, i.e.
\[
	\Ocal_{\CC\times\CC\times\HH}\righteq \tilde{\Po},\quad 1\mapsto \trv:=\frac{1}{J(z,w,\tau)}\otimes (d z)^\vee.
\] 
	The analytification of the universal vectorial extension $\Ecal^\dagger$ of $\Ecal^\vee$ sits in a short exact sequence (cf. \cite[Ch I, 4.4]{mazur_messing})
\[
		\begin{tikzcd}
			0\ar[r] & R^1(\pi_{\Ecal}^{\textit{an}})_*(2\pi i\ZZ)\ar[r] & \HdR{1}{\Ecal/\Mcal}\ar[r] & \Ecal^{\dagger}(\CC)\ar[r] & 0.
		\end{tikzcd}
	\]
	In particular, the pullback of the geometric vector bundle $\HdR{1}{\Ecal/\Mcal}$ to the universal covering $\HH\rightarrow \Mcal(\CC)$ serves as a universal covering of $\Ecal^\dagger(\CC)$. Choosing coordinates on this universal covering is tantamount to choosing a basis of
	\begin{equation}\label{eq_HdR}
		\HdR{1}{\Ecal/\Mcal}^\vee\righteq \HdR{1}{\Ecal^\vee/\Mcal}.
	\end{equation}
	The isomorphism \eqref{eq_HdR} is canonically induced by Deligne's pairing. Let us choose the differentials of the first and second kind $\omega=[\dd w]$ and $\eta=[ \wp(w,\tau) \dd w]$ as generators of $\HdR{1}{\Ecal^\vee/\Mcal}$ and denote the resulting coordinates by $(w,v)$. We can summarize the resulting covering spaces in the following commutative diagram:
	\[
		\begin{tikzcd}
			\CC^2\times\HH\ar[r,"\pr_1"]\ar[d] & \CC\times\HH \ar[d]\\
			\Ecal^{\dagger}(\CC)\ar[r] & \Ecal^\vee(\CC).
		\end{tikzcd}	
	\]
	The pullback of the Poincar\'e bundle $\Po^\dagger$ to $\Ecal\times \Ecal^\dagger$ is equipped with a canonical integrable $\Ecal^\dagger$-connection
	\[
		\con{\dagger}:\Po^\dagger\rightarrow \Po^\dagger\otimes_{\Ocal_{\Ecal\times \Ecal^\dagger}} \Omega^1_{\Ecal\times \Ecal^\dagger/\Ecal^\dagger}.
	\]	
	Let us write $\tilde{\Po}^\dagger$ for the pullback of the Poincar\'e bundle to the universal covering. The trivializing section $\trv^\dagger$ of $\tilde{\Po}$ induces a trivializing section $\trv^\dagger$ of $\tilde{\Po}^\dagger$ via pullback:
	\begin{equation}\label{eq_PoC1}
		\Ocal_{\CC\times\CC^2\times\HH}\righteq \tilde{\Po}^\dagger,\quad 1\mapsto \trv^\dagger.
	\end{equation}
	Let us write $\tilde{\Lcal}^\dagger_1$ for the pullback of $\Lcal^\dagger_1:=\pi_*(\Po^\dagger|_{\Ecal\times_\Mcal\Inf^1_e\Ecal^\dagger})$ along the universal covering
	\[
		\tilde{\pi}\colon\tilde{\Ecal}=\CC\times\HH\twoheadrightarrow	\Ecal(\CC).
	\]	
	Using $\Inf^1_e\Ecal^\dagger=\Ocal_\Mcal\oplus \HdR{1}{\Ecal^\vee/\Mcal}$ the trivialization \eqref{eq_PoC1} induces
	\[
		\triv: \tilde{\Lcal}^\dagger_1 \righteq \Ocal_{\tilde{\Ecal}}\oplus \tilde{\pi}^*\HdR{1}{\Ecal^\vee/\Mcal}.
	\] 
	Recall, that the higher de Rham logarithm sheaves were defined by tensor symmetric powers of the first de Rham logarithm sheaf.	The generators $\omega=[\dd w], \eta=[\wp(w,\tau)\dd w] $ of $\tilde{\Hcal}:=\tilde{\pi}^*\HdR{1}{\Ecal^\vee/\Mcal}$ together with the isomorphism
	\[
		\triv\colon\tilde{\Lcal}_n^\dagger\righteq \TSym^n\tilde{\Lcal}^\dagger_1 \cong \TSym^n (\Ocal_{\tilde{\Ecal}}\oplus \tilde{\Hcal})=\bigoplus_{i=0}^n \TSym^i \tilde{\Hcal}
	\]
	induce a decomposition
	\[
		\triv\colon\tilde{\Lcal}_n^\dagger\righteq \bigoplus_{i+j\leq n} \tilde{\omega}^{[i,j]}\cdot \Ocal_{\tilde{\Ecal}}
	\]
	with $\tilde{\omega}^{[i,j]}:=\triv^{-1}(\omega^{[i]}\cdot \eta^{[j]} )$. Here, $(\cdot)^{[i]}$ denotes the canonical divided power structure on the algebra of tensor symmetric powers.
	\begin{lem}[{\cite[(3.4.16)]{rene}}]\label{lem_explicit_con} In terms of this decomposition the connection $\con{\Ln^\dagger}$ is given by
	\begin{align*}
			\con{\Ln^\dagger}(\tilde{\omega}^{[i,j]})&=-(i+1)\eta_1(\tau)\cdot \tilde{\omega}^{[i+1,j]} \otimes \dd z+(j+1)\tilde{\omega}^{[i,j+1]}  \otimes \dd z
		\end{align*}
	with the convention that $\tilde{\omega}^{[i,j]}=0$ if $i+j>n$. Here, $\eta_1(\tau)$ is the "period of the second kind" $\eta_1(\tau)=\zeta(z,\tau)-\zeta(z+1,\tau)$.
	\end{lem}
	\begin{proof}
		The horizontality of the isomorphism
		\[
			\Ln^\dagger\righteq \TSym^n\Lcal_1^\dagger
		\]	
		reduces us to prove the statement in the case $n=1$. Since the restriction of $\con{\Lcal^\dagger_1}$ to $\Hcal_E$ is trivial, we get
		\[
			\con{\Lcal^\dagger_1}(\tilde{\omega}^{[1,0]})=\con{\Lcal^\dagger_1}(\tilde{\omega}^{[0,1]})=0
		\]
		and it remains to prove
		\[
			\con{\Lcal^\dagger_1}(\tilde{\omega}^{[0,0]})=-\eta_1(\tau)\cdot \tilde{\omega}^{[1,0]} \otimes \dd z+\tilde{\omega}^{[0,1]}  \otimes \dd z.
		\]
		The connection $\con{\Lcal^\dagger_1}$ is induced from the connection $\con{\dagger}$ on $\Po^\dagger$. The explicit description of the connection in \cite[Thm. C.6 (1)]{katz_eismeasure} yields immediately the formula:
		\begin{align*}
			\con{\dagger}(\trv^\dagger)&=\left[-\frac{\partial_z J(z,w)}{J(z,w)} + (\zeta(z+w)-\zeta(z)+v)) \right] \trv^\dagger\otimes \dd z
		\end{align*}
		Using
		\begin{align*}
			\frac{\partial_z J(z,w)}{J(z,w)}&=\partial_z \log J(z,w)=\\
			&=w\cdot\eta_1(\tau)+\zeta(z+w)-\zeta(z)
		\end{align*}
		we get
		\begin{align*}
			\con{\dagger}(\trv^\dagger)&=\left( v- w\cdot\eta_1(\tau)  \right) \trv^\dagger\otimes \dd z.
		\end{align*}
		Restricting this to the first infinitesimal neighborhood gives:
		\[
			\con{\Lcal^\dagger_1}(\tilde{\omega}^{[0,0]})=-\eta_1(\tau)\cdot \tilde{\omega}^{[1,0]} \otimes \dd z+\tilde{\omega}^{[0,1]}  \otimes \dd z.
		\]
	\end{proof}
	\begin{cor}\label{cor_absoluteconnection}
		The absolute connection on $\Lcal_n^\dagger$ is given by the formula
		\begin{align*}
			\con{\Ln^\dagger}^{\text{abs}}(\tilde{\omega}^{[k,j]})&=\Big(-(k+1)\eta_1(\tau)\cdot \tilde{\omega}^{[k+1,j]}+(j+1)\tilde{\omega}^{[k,j+1]}\Big)  \otimes \dd z + \\
			&+\Big( -\frac{\eta_1(\tau)}{2\pi i}k\tilde{\omega}^{[k,j]} + \frac{1}{2\pi i}(j+1)\tilde{\omega}^{[k-1,j+1]}\Big)\otimes d\tau +\\
			&+\Big(\Big(\partial_\tau\eta_1(\tau)
-\frac{\eta_1(\tau)^2}{2\pi i}\Big)(k+1)\tilde{\omega}^{[k+1,j-1]} +\frac{\eta_1(\tau)}{2\pi i}j\tilde{\omega}^{[k,j]}  \Big)\otimes d\tau
		\end{align*}
		with the convention that $\tilde{\omega}^{[i,j]}=0$ if $i+j>n$ or $i,j<0$.
	\end{cor}
	\begin{proof}
		By the horizontality of
		\[
			\Lcal_n^\dagger \righteq \TSym^n\Lcal_1^\dagger
		\]
		it is enough to prove the statement in the case $n=1$. In the following, let us write $\tilde{\nabla}_{\Ln^\dagger}^{\text{abs}}$ for the connection defined by the formulas in the statement. So, in the case $n=1$ these formulas reduce to
		\begin{align*}
			\tilde{\nabla}_{\Lcal_1^\dagger}^{\text{abs}}(\tilde{\omega}^{[0,0]})&= (-\eta_1(\tau)\cdot \tilde{\omega}^{[1,0]}+\tilde{\omega}^{[0,1]})\otimes d z \\
			\tilde{\nabla}_{\Lcal_1^\dagger}^{\text{abs}}(\tilde{\omega}^{[0,1]})&= \left(\partial_\tau\eta_1(\tau)-\frac{\eta_1(\tau)^2}{2\pi i}\right)\tilde{\omega}^{[1,0]}\otimes d \tau +  \frac{\eta_1}{2\pi i}\tilde{\omega}^{[0,1]}\otimes d\tau \\
			\tilde{\nabla}_{\Lcal_1^\dagger}^{\text{abs}}(\tilde{\omega}^{[1,0]})&= -\frac{\eta_1(\tau)}{2\pi i}\tilde{\omega}^{[1,0]}\otimes d\tau+\frac{1}{2\pi i}\tilde{\omega}^{[0,1]}\otimes d\tau.
		\end{align*}
		A straightforward calculation shows that these formulas define an integrable holomorphic connection on $\Lcal_1^\dagger$. Let us now verify, that $(\Lcal_1^\dagger,\tilde{\nabla}_{\Lcal_1^\dagger}^{\text{abs}})$ represents the extension class of the first logarithm sheaf. The logarithm sheaf splits after pullback along $e$:
		\[
			e^*\Lcal_1^\dagger\righteq \Ocal_\Mcal\oplus\Hcal.
		\]
		The section $e^*\tilde{\omega}^{[0,0]}$ is a generator of $\Ocal_\Mcal$ while $e^*\tilde{\omega}^{[0,1]}=\eta$ and $e^*\tilde{\omega}^{[1,0]}=\omega$ form a basis of $\Hcal$. Let us first check, that this splitting is horizontal if we equip the left hand side with $e^*\tilde{\nabla}_{\Lcal_1^\dagger}^{\text{abs}}$, $\Ocal_\Mcal$ with the derivation $d\colon \Ocal_\Mcal \rightarrow \Omega^1_\Mcal$ and $\Hcal$ with the Gauss--Manin connection $\nabla_{\mathrm{GM}}$. The Gauss--Manin connection is given by the formulas (see for example \cite[A1.3.8]{katz_padic_properties})
		\begin{align*}
			\nabla_{\mathrm{GM}}(\eta)&= \left(\partial_\tau\eta_1(\tau)-\frac{\eta_1(\tau)^2}{2\pi i}\right)\omega\otimes d \tau +  \frac{\eta_1}{2\pi i}\eta\otimes d\tau \\
			\nabla_{\mathrm{GM}}(\omega)&= -\frac{\eta_1(\tau)}{2\pi i}\omega\otimes d\tau+\frac{1}{2\pi i}\eta\otimes d\tau.
		\end{align*}
		Comparing this to the defining formulas for $\tilde{\nabla}_{\Lcal_1^\dagger}^{\text{abs}}$ shows immediately the horizontality of the splitting along $e$. Stated differently, we have shown that the extension class $[(\Lcal_1^\dagger, \tilde{\nabla}_{\Lcal_1^\dagger}^{\text{abs}})]$ maps to zero under the map:
		\[
			\Ext^1_{\VIC{\Ecal/\CC}}(\Ocal_\Ecal, \Hcal_\Ecal)\xrightarrow{e^*} \Ext^1_{\VIC{\Mcal/\CC}}(\Ocal_\Mcal, \Hcal).
		\]
		It remains to show, that $[(\Lcal_1^\dagger, \tilde{\nabla}_{\Lcal_1^\dagger}^{\text{abs}}))]$ maps to $\id_\Hcal$ under the map
		\[
			\Ext^1_{\VIC{\Ecal/\CC}}(\Ocal_\Ecal, \Hcal_\Ecal) \rightarrow \Hom_{\VIC{\Mcal/\CC}}(\Ocal_\Mcal, \Hcal\otimes\Hcal^\vee)
		\]		
		appearing in the defining property of the first logarithm sheaf in \Cref{def_log}. Restricting the absolute connection relative $\Mcal$ gives us vertical maps making the diagram 
		\[
			\begin{tikzcd}
				\Ext^1_{\VIC{\Ecal/\CC}}(\Ocal_\Ecal, \Hcal_\Ecal) \ar[r]\ar[d,"\res_\Mcal"] & \Hom_{\VIC{\Mcal/\CC}}(\Ocal_\Mcal, \Hcal\otimes\Hcal^\vee) \ar[d,hook]\\
				\Ext^1_{\VIC{\Ecal/\Mcal}}(\Ocal_\Ecal, \Hcal_\Ecal) \ar[r]  & \Hom_{\Ocal_\Mcal}(\Ocal_\Mcal, \Hcal\otimes\Hcal^\vee)
			\end{tikzcd}
		\]
		commute. The defining formulas for $\tilde{\nabla}_{\Lcal_1^\dagger}^{\text{abs}}$ together with \Cref{lem_explicit_con} yield $\res_\Mcal(\tilde{\nabla}_{\Lcal_1^\dagger}^{\text{abs}})=\nabla_{\Lcal_1^\dagger}$. But we already know that $\nabla_{\Lcal_1^\dagger}=\res_\Mcal(\nabla_{\Lcal_1^\dagger}^{\text{abs}})$, hence $(\Lcal_1^\dagger, \nabla_{\Lcal_1^\dagger})$ maps to $\id_\Hcal$ under the lower horizontal morphism in the above diagram. Since the left vertical map is injective, we deduce that $(\Lcal_1^\dagger, \tilde{\nabla}_{\Lcal_1^\dagger}^{\text{abs}})$ maps to $\id_\Hcal$ under the upper horizontal map.
	\end{proof}
	In terms of our trivializing section $\trv^\dagger$ the Kronecker section $\scan$ expresses as follows:
	\[
		\scan=J(z,w,\tau)\cdot\trv^\dagger\otimes d z.
	\]
	This implies the following formula for the $D$-variant of the Kronecker section:
	\[
		\scan^D=\Big( D^2J(z,Dw,\tau)-DJ(Dz,w,\tau) \Big)\cdot \trv^\dagger\otimes d z.
	\]
	The expansion coefficients
	\[
		D^2J(z,w,\tau)-DJ(Dz,\frac{w}{D},\tau)=\sum_{k=0}^\infty s_k^D(z,\tau)w^k
	\]
	allow us to describe the restriction of $\scan^D$ to the $n$-th infinitesimal neighborhood along $\Ecal$ as
	\[
		\lnD=\sum_{k=0}^n k!s_k^D(z,\tau)\tilde{\omega}^{[k,0]}\otimes dz. 
	\]
	The Kodaira--Spencer isomorphism identifies $dz\otimes dw$ with $\frac{1}{2\pi i}d\tau$ so we get
	\begin{equation}\label{eq_rene}
		L_n^D=\sum_{k=0}^n \left(k!s_k^D(z,\tau)\tilde{\omega}^{[k,0]}\otimes dz + \frac{1}{2\pi i} (k+1)!s_{k+1}^D(z,\tau)\tilde{\omega}^{[k,0]}\otimes d\tau\right). 
	\end{equation}
	In particular, we deduce that the analytification of the $1$-forms $L_n^D$ coincide with the analytic $1$-forms used by Scheider to describe the de Rham realization of the elliptic polylogarithm on the universal elliptic curve. The analytic expression \eqref{eq_rene} is exactly the analytic section of the de Rham logarithm sheaves which was used by Scheider to describe the de Rham realization of the elliptic polylogarithm analytically. We have reduced the purely algebraic statement of \Cref{DR_liftingLnD} to the analytification of the modular curve and identified the objects with the analytic description of Scheider. Thus from here on we can follow the argument in Scheider \cite[Thm. 3.6.2]{rene}. For the convenience of the reader let us nevertheless finish the proof. Indeed, it will be the mixed heat equation 
		\[
		2\pi i \cdot \partial_\tau J(z,w,\tau)=\partial_z\partial_w J(z,w,\tau).
	\]
	 which will be responsible for the vanishing of $\LnD$ under the differential in the de Rham complex. The mixed heat equation implies the formula
	\begin{align}\label{heat_eq}
		\partial_\tau s_k^D=\frac{1}{2\pi i}(k+1)\partial_z s_{k+1}^D
	\end{align}
	and we compute
	\begin{align*}
		&(\con{\Ln^\dagger}^{\text{abs}}\wedge \id +\id\otimes d)(L_n^D)=\\
		=&(\con{\Ln^\dagger}^{\text{abs}}\wedge \id +\id\otimes d)( \sum_{k=0}^n \left(k!s_k^D\tilde{\omega}^{[k,0]}\otimes dz + \frac{1}{2\pi i} (k+1)!s_{k+1}^D\tilde{\omega}^{[k,0]}\otimes d\tau\right)  )=\\
		=& \sum_{k=0}^n \left( k!\partial_\tau s_k^D \tilde{\omega}^{[k,0]}\otimes d\tau \wedge dz +  \frac{1}{2\pi i}(k+1)!\partial_z s_{k+1}^D \tilde{\omega}^{[k,0]}\otimes d z \wedge d\tau   \right)+\\
		& + \sum_{k=0}^n k! s_k^D \con{\Ln^\dagger}^{\text{abs}}(\tilde{\omega}^{[k,0]})\otimes d z + \sum_{k=0}^n \frac{1}{2\pi i} (k+1)! s_{k+1}^D \con{\Ln^\dagger}^{\text{abs}}(\tilde{\omega}^{[k+1,0]})\otimes d \tau\stackrel{\eqref{heat_eq}}{=}\\
		=& \sum_{k=0}^n k! s_k^D \con{\Ln^\dagger}^{\text{abs}}(\tilde{\omega}^{[k,0]})\otimes d z + \sum_{k=0}^n \frac{1}{2\pi i} (k+1)! s_{k+1}^D \con{\Ln^\dagger}^{\text{abs}}(\tilde{\omega}^{[k+1,0]})\otimes d \tau\stackrel{\text{Cor. }\ref{cor_absoluteconnection}}{=}\\
	\end{align*}
	
	\begin{align*}
		=& \sum_{k=0}^n \frac{k!}{2\pi i} s_k^D\cdot \left( -\eta_1(\tau)\cdot k\cdot \tilde{\omega}^{[k,0]}+\tilde{\omega}^{[k-1,1]}  \right) \otimes d\tau\wedge d z +\\
		& + \sum_{k=0}^n \frac{(k+1)!}{2\pi i}  s_{k+1}^D\cdot  \left( -\eta_1(\tau)\cdot (k+1)\cdot \tilde{\omega}^{[k+1,0]}+\tilde{\omega}^{[k,1]}  \right)\otimes dz\wedge d \tau=\\
		=& \sum_{k=1}^n \frac{k!}{2\pi i} s_k^D\cdot \left( -\eta_1(\tau)\cdot k\cdot \tilde{\omega}^{[k,0]}+\tilde{\omega}^{[k-1,1]}  \right) \otimes d\tau\wedge d z -\\
		& - \sum_{k=1}^n \frac{k!}{2\pi i}  s_{k}^D\cdot  \left( -\eta_1(\tau)\cdot k \cdot \tilde{\omega}^{[k,0]}+\tilde{\omega}^{[k-1,1]}  \right)\otimes d\tau\wedge d z=0\\
	\end{align*}
	Thus $\LnD$ is a closed form with respect to the differential of the de Rham complex of $\Ln^\dagger$ and the proof of the Proposition is finished.

\section{The de Rham Eisenstein classes} The aim of this section is to describe the de Rham Eisenstein classes explicitly. We will identify them with cohomology classes associated to certain Eisenstein series. In the following we will use $(\Ln^\dagger,\nabla_{\Ln^\dagger}^{\mathrm{abs}})$ as an explicit model for the de Rham logarithm sheaves. The canonical horizontal isomorphism
\[
	\Ln^\dagger\righteq \TSym^n \Lcal_1^\dagger
\]
together with the horizontal isomorphism
\[
	e^*\Lcal_1^\dagger \righteq \Ocal_S\oplus \Hcal
\]
induces a splitting isomorphism
\[
	e^*\Ln^\dagger\righteq \prod_{k=0}^n \TSym^k_{\Ocal_S} \Hcal.
\]
It might be more common to work with $\Sym^k \HdR{1}{E/S}$ instead of $\TSym^k \HdR{1}{\Ed/S}$. Thus, let us make the following identifications: By the universal property of the symmetric algebra, we have a canonical ring homomorphism
\[
	\Sym^\bullet \Hcal \rightarrow \TSym^\bullet \Hcal,
\]
which is an isomorphism since we are working over a field of characteristic zero. Further, let us use the polarization $E\righteq \Ed$ associated with the ample line bundle $\Ocal_E([e])$ to identify
\[
	\Hcal=\HdR{1}{\Ed/S}\righteq \HdR{1}{E/S}.
\]
With these identifications, we can write the above splitting isomorphism as follows:
\[
	e^*\Lcal_n^\dagger \righteq \prod_{k=0}^n \Sym^k_{\Ocal_S} \HdR{1}{E/S}.
\]
Similarly, we have
\[
	e^*\Lcal_n \righteq \prod_{k=0}^n \om_{E/S}^{\otimes k}.
\]
Further, by invariance under isogenies we have an isomorphism
\[
	\Ln^\dagger\righteq [N]^*\Ln^\dagger.
\]
For a torsion section $s\in E[N](S)$ we get a canonical horizontal isomorphism
\[
	T_s^*\Ln^\dagger \righteq T_s^*[N]^*\Ln^\dagger=[N]^*\Ln^\dagger\righteq \Ln^\dagger.
\]
where $T_s:E\rightarrow E$ is the translation by $s$. Together with the splitting isomorphism we obtain a horizontal isomorphism:
\[
	\triv_s\colon s^*\Ln^\dagger=e^*T_s^*\Ln^\dagger\righteq e^*\Ln^\dagger \righteq \prod_{k=0}^n \Sym^k\HdR{1}{E/S}.
\]
The trivialization map is compatible with the Hodge filtration, i.e. we have
\[
	\triv_s\colon s^*\Ln=e^*T_s^*\Ln\righteq e^*\Ln \righteq \prod_{k=0}^n \om_{E/S}^{\otimes k}.
\]
The map $\triv_s$ induces the \emph{specialization map}:
\[
	s^*\colon \HdRabs{1}{U_D,\Ln^\dagger}\rightarrow \prod_{k=0}^n\HdRabs{1}{S,\Sym^k\HdR{1}{E/S}}.
\]
The aim of this section is to identify $s^*\polD^n$ with cohomology classes of certain Eisenstein series: Let us consider the following analytic Eisenstein series
\[
	F^{(k)}_{(a,b)}(\tau)=(-1)^{k+1}(k-1)!\sum_{(0,0)\neq(m,n)\in\ZZ^2} \frac{1}{(m\tau+n)^k}\zeta_N^{mb-na},\quad \zeta_N:=\exp(\frac{2\pi i}{N}).
\]
and define
\[
	\FKatoD{k}{(a,b)}(\tau)=D^2 F^{(k)}_{(a,b)}(\tau)- D^{1-k}F^{(k)}_{(Da,Db)}(\tau).
\]
These are exactly the Eisenstein series appearing in Kato's Euler system c.f.~\cite[\S 3.6]{kato}. Let $\Ecal/\Mcal$  be the universal elliptic curve over $\QQ$ with $\Gamma(N)$-level structure. Recall that modular forms of level $\Gamma(N)$ and weight $k$ are exactly the sections of $\Gamma(\Mcal,\om_{\Ecal/\Mcal}^{\otimes k})$ which are finite at the cusps. The Kodaira--Spencer map
\[
	\om_{\Ecal/\Mcal}^{\otimes 2} \righteq \Omega^1_{\Mcal}
\]
allows us to associate de Rham cohomology classes to modular forms of weight $k\geq 2$ via:
\[
	\Gamma(\Mcal,\om_{\Ecal/\Mcal}^{\otimes k})\righteq \Gamma(\Mcal,\om_{\Ecal/\Mcal}^{\otimes (k-2)}\otimes \Omega^1_{\Mcal})\stackrel{[\cdot]}{\rightarrow} \HdRabs{1}{\Mcal, \Sym^{k-2}\HdR{1}{\Ecal/\Mcal}}.
\]
For a modular form $f$ of weight $k$ let us write $[f]\in \HdRabs{1}{\Mcal, \Sym^{k-2}\HdR{1}{\Ecal/\Mcal}}$  for its associated cohomology class. The explicit description of the de Rham polylogarithm via the Kronecker section allows us to deduce an explicit formula for the de Rham Eisenstein classes. The de Rham Eisenstein classes have been known previously, see for example \cite[Prop. 3.8]{bannai_kings}, \cite[Thm. 3.8.15]{rene}.
\begin{thm}
Let $\Ecal/\Mcal$ be the universal elliptic curve over $\QQ$ with $\Gamma(N)$-level structure. Let $(0,0)\neq(a,b)\in \left(\ZZ/N\ZZ\right)^2$ and $s=s_{(a,b)}$ be the associated $N$-torsion section. The $D$-variant of the polylogarithm specializes to the following cohomology classes of Eisenstein series:
\[
	s^*\polD^n=\left(\left[\frac{\FKatoD{k+2}{(a,b)}}{k!}\right]\right)_{k=0}^{n}
\]
\end{thm}
\begin{proof}
By \Cref{DR_Pol_thm} it suffices to prove
\[
	s^*[\LnD]=\left(\left[\frac{\FKatoD{k+2}{(a,b)}}{k!}\right]\right)_{k=0}^{n}.
\]
Since $\LnD$ is obtained by applying
\begin{equation}
	\Lcal_{n+1}\otimes\Omega^1_{\Ecal/\Mcal}\rightarrow \Lcal_{n}\otimes\Lcal_1\otimes\Omega^1_{\Ecal/\Mcal} \rightarrow \Lcal_{n}\otimes \Omega^1_{\Ecal/\QQ}.
\end{equation}
to $l_{n+1}^D$, we are reduced to prove
\[
	\triv_s(s^*l_{n+1}^D)=\left(\frac{\FKatoD{k+1}{(a,b)}}{k!}\right)_{k=0}^{n+1}.
\]
Our aim is to reduce this claim to the construction of Eisenstein--Kronecker series via the Poincar\'e bundle. As a first step we have to compare the translation operators of the Poincar\'e bundle to the translation operators of the logarithm sheaves. Let us recall the definitions: The definition of the translation isomorphism 
\[
	T_s^*\Ln^\dagger \righteq T_s^*[N]^*\Ln^\dagger=[N]^*\Ln^\dagger\righteq \Ln^\dagger. 
\]
involves twice the invariance under isogenies isomorphism $[N]^*\Ln^\dagger$ which in turn is induced by restricting $\gamma_{\id,[N]}\colon (\id\times[N])^*\Po^\dagger\righteq ([N]\times \id)^*\Po^\dagger$ along $E\times \Inf^n E^\dagger$. More generally, one can define $\gamma_{[N],[D]}$ as the diagonal in the commutative diagram
\[
\begin{tikzcd}[column sep=huge]
	([N]\times [D])^{*}\Po \ar[r,"({[N]}\times \id)^{*}\gamma_{\id,[D]}"]\ar[d,swap,"(\id\times {[D]})^{*}\gamma_{[N],\id}"] & ([ND]\times\id)^{*}\Po \ar[d,"({[D]}\times \id)^{*}\gamma_{[N],\id}"] \\
	(\id\times[DN])^{*}\Po \ar[r,"(\id\times {[N]})^{*}\gamma_{\id,[D]}"] & ([D]\times[N])^{*}\Po.
\end{tikzcd}
\]
Using this, we have defined in \cite[\S 3.3]{EisensteinPoincare} translation isomorphisms $\Ucal^{N,D}_{s,t}$ for $s\in \Ecal[N](\Mcal)$ and $t\in \Ecal^\vee[D](\Mcal)$ on the Poincar\'e bundle $\Po$:
\[
\begin{tikzcd}
\Ucal^{N,D}_{s,t}:=\gamma_{[N],[D]}\circ(T_{s}\times T_t)^*\gamma_{[D],[N]}: (T_s\times T_t)^*([D]\times [N])^*\Po \ar[r] &  ([D]\times [N])^*\Po.
\end{tikzcd}
\]
For sections $\sigma\in\Gamma(U,\Po)$ it is convenient to introduce the notation $U_{s,t}^{N,D}(\sigma):= \Ucal^{N,D}_{s,t}((T_s\times T_t)^*([D]\times [N])^*\sigma)$. Let $\sigma\in\Gamma(U,\Po)$ be a section of the Poincar\'e bundle and $s\in E[N](S)$ and $t\in E^\vee[D](S)$. In particular, we get the section
\[
	\sigma_t:=(\pr_E)_*\left( U^{1,D}_{e,t}(\sigma)\Big|_{E\times_S\Inf^n_e\Ed} \right)
\]
of the geometric logarithm sheaf $[D]^*\Ln$. Similarly, by taking translates by $s$ and $t$ we get the section
\[
	\sigma_{s,t}:=(\pr_E)_*\left( (\id\times [N]^\sharp)^{-1}\left[U^{N,D}_{s,t}(\sigma)\Big|_{E\times_S\Inf^n_e\Ed}\right] \right)
\]
of $[D]^*\Ln$. Here, we wrote $[N]^\sharp$ for the isomorphism of structure sheaves $\Ocal_{\Inf^n_e\Ed}\righteq [N]^*\Ocal_{\Inf^n_e\Ed}$ induced by $N$-multiplication. Let us write
\[
	\inv_{[D]}\colon \Ln\righteq [D]^*\Ln
\]
for the invariance under isogenies isomorphism. Unwinding the definitions it is straightforward to check that the translation operators of the Poincar\'e bundle and the translation operators of the geometric logarithm sheaves are compatible in the following precise sense:
\[
	\trans_s\left( T_s^* \inv_{[D]}^{-1}(\sigma_t) \right)=\inv^{-1}_{[D]}(\sigma_{s,t}).
\]
Applying this to the Kronecker section $\sigma=\scan$ gives us sections
\[
	\sigma_t\in\Gamma\left(\Ecal,[D]^*\Ln\otimes \Omega^1_{E/S}(E[D]))\right)
\]
and
\[
	\sigma_{s,t}\in\Gamma\left(\Ecal,[D]^*\Ln\otimes \Omega^1_{E/S}(T_s^*E[D]))\right).
\]
By the distribution relation \cite[Corollary A.3]{EisensteinPoincare}
	\[
			\sum_{e\neq t\in \Ed[D](S)} U^{D}_{t}(\scan)=(D)^2\cdot \gamma_{1,D}\left( (\id \times [D])^*(\scan)\right)-([D]\times \id)^*(\scan).
	\]
 and the definition of $\lnD$ we obtain
\[
	\lnD=\sum_{e\neq t\in \Ed[D]}(\inv_{[D]}^{-1}\otimes\mathrm{can})(\sigma_t)
\]
where $\mathrm{can}\colon [D]^*\Omega^1_{\Ecal/\Mcal}([e])\righteq  \Omega^1_{\Ecal/\Mcal}(\Ecal[D])$ is the canonical isomorphism. Now, the above formula gives
\begin{equation}\label{eq_trivlnD}
	(\trans_s\otimes\id_{\Omega^1_{\Ecal/\Mcal}})\left( T_s^* \lnD \right)=(\inv_{[D]}\otimes \mathrm{can})^{-1} \left(\sum_{e\neq t\in \Ecal^\vee[D]}\sigma_{s,t}\right).
\end{equation}
In particular, we obtain the formula
\begin{align*}
	(\triv_s\otimes{\id_{\omega_{\Ecal/\Mcal}}})(s^*\lnD)\stackrel{\mathrm{Def.}}{=}&(e^*\triv_e\otimes{\id_{\omega_{\Ecal/\Mcal}}})\circ (e^*\trans_s\otimes{\id_{\omega_{\Ecal/\Mcal}}}) \left( e^*T_s^*\lnD \right)\stackrel{\eqref{eq_trivlnD}}{=}\\
	=&(e^*\triv_e\otimes{\id_{\omega_{\Ecal/\Mcal}}})\circ(e^*\inv_{[D]}\otimes e^*\mathrm{can})^{-1}\left(\sum_{e\neq t\in \Ecal^\vee[D]}e^*\sigma_{s,t}\right)\stackrel{(A)}{=}\\
	=& \left(\bigoplus_{k=0}^{n}(\cdot D^{-k-1})\right)\circ (e^*\triv_e)\left(\sum_{e\neq t\in \Ecal^\vee[D]}e^*\sigma_{s,t}\right)\stackrel{(B)}{=}\\
	=& \left(\frac{D^{-k+1}}{k!}\cdot \sum_{e\neq t\in \Ecal^\vee[D]}(e\times e)^*\left[ ([D]\times[N])^*\nabla_\sharp^{\circ k}U_{s,t}^{N,D}(\scan) \right]\right)_{k=0}^{n}\stackrel{(C)}{=}\\
	=& \left(\frac{D^{-k+1}}{k!}\cdot \sum_{e\neq t\in \Ecal^\vee[D]} E^{k,1}_{s,t}\right)_{k=0}^n
\end{align*}
Here, $(A)$ is induced by the commutativity of
\[
	\begin{tikzcd}[column sep=huge]
		e^*\Ln\otimes e^*\Omega^1_{\Ecal/\Mcal}\ar[d,"\triv_e"] \ar[r,"e^*\inv_{[D]}\otimes e^*\mathrm{can}^{-1}"] & e^*[D]^*\Ln\otimes e^*[D]^*\Omega^1_{\Ecal/\Mcal}\ar[d,"\triv_e"]\\
		\bigoplus_{k=0}^n \om_{\Ecal/\Mcal}^{\otimes k}\otimes\om_{\Ecal/\Mcal} \ar[r,"({[D]^*})^{\otimes k}\otimes({[D]^*})^{-1}"]  & \bigoplus_{k=0}^n \om_{\Ecal/\Mcal}^{\otimes k}\otimes\om_{\Ecal/\Mcal}.
	\end{tikzcd}
\]
and the fact that pullback along $[D]$ induces multiplication by $D$ on the co-tangent space $\om_{\Ecal/\Mcal}$. For the equality $(B)$, let us recall that $\con{\sharp}$ is the universal integrable $\Ecal^\sharp$-connection on the pullback $\Po^\sharp$ of the Poincar\'e bundle to $\Ecal^\sharp\times_\Mcal \Ecal^\vee$, where $\Ecal^\sharp$ denotes the universal vectorial extension of $\Ecal$. The $\Ocal_\Mcal$-linear map 
\[
	(e\times \id_{\Ecal^\vee})^*\con{\sharp}\colon \Ocal_{\Ecal^\vee}\rightarrow \Ocal_{\Ecal^\vee}\otimes_{\Ocal_{\Ecal^\vee}} \Omega^1_{\Ecal^\vee/\Mcal}=\Ocal_{\Ecal^\vee}\otimes_{\Ocal_{\Mcal}}\om_{\Ecal/\Mcal}
\]
is nothing than the invariant derivation on $\Ecal^\vee$. On the other hand, the map
\[
	\triv_e\colon \Ocal_{\Inf^n_e\Ecal^\vee}=(e\times\id_{\Inf^n_e\Ecal^\vee})^*\left(\Po\Big|_{\Ecal\times \Inf^n_e\Ecal^\vee}\right) \righteq \bigoplus_{k=0}^n \om_{\Ecal/\Mcal}^{\otimes k}
\]
coincides with $f\mapsto (e^*(\partial^{\circ k} f)/k!)_{k=0}^n$, i.e. it is given by iteratively applying the invariant derivation
\[
	\partial\colon \Ocal_{\Inf^n_e\Ecal^\vee} \rightarrow \Ocal_{\Inf^{n-1}_e\Ecal^\vee}\otimes_{\Ocal_\Mcal}\om_{\Ecal/\Mcal}
\]
to sections of $\Ocal_{\Inf^n_e\Ecal^\vee}$. Combining these two facts with the definition of $\sigma_{s,t}$ gives $(B)$. The equality $(C)$ is the definition of the geometric modular forms $E^{k,1}_{s,t}$ given in \cite[Def. 4.1]{EisensteinPoincare}.\par 
So far, we have proven that the specialization $s^*\polD^n$ is represented by the cohomology classes associated to the geometric modular forms $D^{-k-1}\cdot E^{1,k}_{s,t}$. It remains to relate $D^{-k-1}\cdot E^{1,k}_{s,t}$ to Kato's Eisenstein series $\FKatoD{k}{(a,b)}$. Let us compare $D^{-k-1}\cdot \sum_{e\neq t\in \Ed[D]}E^{1,k}_{s,t}$ to $\FKatoD{k}{(a,b)}$ on the analytification of the universal elliptic curve.

Let $\Ecal/\Mcal$ be the universal elliptic curve of level $\Gamma_1(N)$. Let us choose the following explicit model for the analytification
\[
	\Ecal(\CC)= ((\ZZ/N\ZZ)^\times \times\CC \times \HH)/  (\ZZ^2\rtimes \Gamma_1(N))
\]
with coordinates $(j,z,\tau)$ on $(\ZZ/N\ZZ)^\times \times\CC \times \HH$. We will use the trivializing section $d z$ of $\om_{\Ecal(\CC)/\Mcal(\CC)}$ to identify classical modular forms with sections of $\om_{\Ecal(\CC)/\Mcal(\CC)}^{\otimes k}$. According to \cite[Thm 4.2]{EisensteinPoincare} and the functional equation of Eisenstein--Kronecker series the geometric modular form $D^{-k-1}\cdot \sum_{e\neq t\in \Ed[D]} E^{1,k}_{s,t}$ corresponds to the classical modular form
\begin{align*}
	&(-1)^{k}k!\cdot D^{-k+1}\sum_{(0,0)\neq(c,d)\in (\ZZ/D\ZZ)^2}\sum_{(0,0)\neq(m,n)\in\ZZ^2} \frac{1}{(m\tau+n+\frac{c\tau}{D}+\frac{d}{D})^{k+1}}\zeta_N^{(mb-na)}=\\
	=&(-1)^{k}k!\cdot D^{-k+1}\sum_{(m,n)\in\ZZ^2\setminus (D\ZZ)^2} D^{k+1} \frac{1}{(m\tau+n)^{k+1}}\zeta_N^{(mb-na)}=\\
	=&D^2 F^{(k+1)}_{(a,b)}(\tau)- D^{2-(k+1)}F^{(k+1)}_{(Da,Db)}(\tau)=\FKatoD{k+1}{(a,b)}(\tau)
\end{align*}
This proves the desired formula
\[
	s^*\polD^n=\left(\left[\frac{\FKatoD{k+2}{(a,b)}}{k!}\right]\right)_{k=0}^{n}.
\]
\end{proof}

\bibliographystyle{amsalpha} 
\bibliography{PaperdeRham}
\end{document}